\documentclass[letterpaper]{article}

\usepackage{amsmath}
\usepackage{amsfonts}
\usepackage{amsthm}
\usepackage{amssymb}
\usepackage{datetime}
\usepackage{color}

\theoremstyle{plain}
\newtheorem{theorem}{Theorem}[section]
\newtheorem{lemma}[theorem]{Lemma}
\newtheorem{corollary}[theorem]{Corollary}
\theoremstyle{definition}
\newtheorem{definition}[theorem]{Definition}

\newtheorem{example}[theorem]{Example}
\newtheorem{remark}[theorem]{Remark}

\numberwithin{equation}{section}

\newcommand{\all}{\hbox{for all}}

\newcommand{\bra}[2]{\langle#1,#2\rangle}

\newcommand{\Bra}[2]{\big\langle#1,#2\big\rangle}

\newcommand{\dbs}{^{**}}
\newcommand{\dist}{\hbox{\rm dist}}
\newcommand{\dom}{\hbox{\rm dom}}

\newcommand{\eps}{\varepsilon}

\newcommand{\F}{{\mathbb F}}

\newcommand{\half}{{\textstyle\frac{1}{2}}}

\newcommand{\I}{\mathbb I}
\newcommand{\ifff}{\Longleftrightarrow}
\newcommand{\IKT}{{\I_\KT}}
\newcommand{\infn}{\inf\nolimits}
\newcommand{\intr}{\hbox{\rm int}}

\newcommand{\KT}{{\wt K}}
\newcommand{\limn}{\lim\nolimits}

\newcommand{\lr}{\Longrightarrow}

\newcommand{\minn}{\min\nolimits}
\newcommand{\NN}{\mathbb N}

\newcommand{\pIK}{{\partial\I_K}}
\newcommand{\ptKT}{{\partial\tauKT}}

\newcommand{\PCLSC}{{\cal PCLSC}}

\newcommand{\qlr}{\quad\Longrightarrow\quad}

\newcommand{\quand}{\quad\hbox{and}\quad}
\newcommand{\rbar}{\,]{-}\infty,\infty]}

\newcommand{\RR}{\mathbb R}

\newcommand{\supn}{\sup\nolimits}

\newcommand{\tauKT}{{\tau_\KT}}
\newcommand{\tauWT}{{\tau_{-\WT}}}
\newcommand{\toto}{\rightrightarrows}

\newcommand{\ts}{\textstyle}
\newcommand{\UT}{{\wt U}}
\newcommand{\VT}{{\wt V}}

\newcommand{\wh}{\widehat}
\newcommand{\wt}{\widetilde}
\newcommand{\WT}{{\wt W}}

\newcommand{\Cor}{Corollary~\ref}
\newcommand{\Def}{Definition~\ref}

\newcommand{\Ex}{Example~\ref}
\newcommand{\Exs}{Examples~\ref}
\newcommand{\Lem}{Lemma~\ref}
\newcommand{\Lems}{Lemmas~\ref}

\newcommand{\Rem}{Remark~\ref}
\newcommand{\Sec}{Section~\ref}

\newcommand{\Thm}{Theorem~\ref}
\newcommand{\Thms}{Theorems~\ref}

\title{Quasidense monotone multifunctions}
\author{
Stephen Simons
\thanks{
Department of Mathematics, University of California, Santa Barbara, CA\ 93106-3080, U.S.A.
Email: \texttt{stesim38@gmail.com}.}}

\date{}

\begin{document}

\maketitle

\begin{abstract}\noindent
In this paper, we discuss quasidense multifunctions from a Banach space into its dual, and use the two sum theorems proved in a previous\break paper to give various characterizations of quasidensity.   We investigate the Fitzpatrick extension of such a multifunction. We prove that, for closed monotone multifunctions,  quasidensity implies type (FPV) and strong maximality, and that quasidensity is equivalent to type (FP).   
\end{abstract}




\section{Introduction}
This is a sequel to the paper \cite{PARTONE}, in which we introduced the concepts of {\em Banach SN space, $L$--positive set, $r_L$--density} and {\em Fitzpatrick extension}.   In this paper we suppose that $E$ is a nonzero real Banach space with dual $E^*$, and we apply some of the results of \cite{PARTONE} to the Banach SN space $E \times E^*$.   In this case, {\em $L$--positive} means the same as {\em monotone} and we use the word {\em quasidense} to stand for {\em $r_L$--dense}.
\par
In \Sec{BANsec}, we introduce some general Banach space notation, and give the concepts and results from \cite{PARTONE} that we shall use.   Let $S\colon\ E \toto E^*$ be a multi-function.   We define the {\em quasidensity} of $S$ in \Def{QDdef}.   We point out in\break \Thm{RLMAXthm} and \Ex{TAILex} that every closed, monotone quasidense\break multifunction is maximally monotone, but that there exists maximally monotone linear operators that are not quasidense.    We define the function $\varphi_S$ in\break \Def{PHdef}. ($\varphi_S$ is the ``Fitzpatrick function'' of $S$, which was originally\break introduced in the Banach space setting in \cite[(1988)]{FITZ}, but lay dormant until it was rediscovered by  Mart\'\i nez-Legaz and Th\'era in \cite[(2001)]{MLT}.  It had been previously considered in the finite--dimensional setting by Krylov in \cite{KRYLOV}.) In\break \Thm{PHTHthm}, we give a criterion in terms of the Fenchel conjugate, ${\varphi_S}^*$, of  $\varphi_S$ for a maximally monotone multifunction to be quasidense (other criteria can be found in \Thms{FUZZDthm}, \ref{FUZZEthm} and \ref{FPthm}). In \Thm{RTRthm} we prove that the subdifferential of a proper, convex, lower semicontinuous function on $E$ is quasidense (thus generalizing Rockafellar's result, \cite[(1970)]{RTRMMT}).   The rest of \Sec{BANsec} is devoted to the discussion of two significant examples that will be used in later parts of the paper.
\par
If $S\colon\ E \toto E^*$ is closed, monotone and quasidense then, in \Sec{FITZDEFsec}, we\break define an associated maximally monotone multifunction $S^\F\colon\ E^* \toto E\dbs$, which we call the {\em Fitzpatrick extension} of $S$.   $S^\F$ is defined formally in terms of ${\varphi_S}^*$ in \Def{FITZdef}, and we give three characterization of $S^\F$ in \Thm{SEQCHARthm} and two in \eqref{NI3} and \eqref{NI4}. \big(It is observed in Appendix 1, \Sec{APPsec}, that\break $(y^*,y\dbs) \in G(S^\F)$ exactly when $(y\dbs,y^*)$ is in the {\em Gossez extension of $G(S)$}\big).   Now let $f$ be a proper, convex, lower semicontinuous function on $E$.   We prove in \Lem{FITZDIFFlem} that $G\big((\partial f)^\F\big) \subset G\big(\partial(f^*)\big)$.   \Lem{FITZDIFFlem} will be used several times in subsequent parts of the paper.   In \Thm{FITZDIFFthm}, we strengthen \Lem{FITZDIFFlem} and prove that, in fact, $(\partial f)^\F = \partial(f^*)$.    \Thm{FITZDIFFthm} will be used in\break \Thms{FITZPARTAUthm} and \ref{PHINKthm} to characterize the Fitzpatrick extensions of the two\break examples introduced in \Sec{BANsec}.
\par
In \Sec{SUMSsec}, we state the {\em Sum theorem with domain constraints} and the {\em Sum theorem with range constraints} that were established in \cite{PARTONE}.
\par
In \Sec{FPVsec}, we prove that every closed, monotone quasidense multifunction is of {\em type (FPV)}.   (Type (FPV) was introduced  independently by Fitzpatrick--Phelps and Verona--Verona in \cite[p.\ 65(1995)]{FITZTWO} and \cite[p.\ 268(1993)]{VERONAS} under the name of ``maximal monotone locally''.)
\par
In \Sec{FUZZsec}, we give two ``fuzzy'' criteria for quasidensity in which an element of $E^*$ is replaced by a nonempty $w(E^*,E)$--compact convex subset of $E^*$, or an element of $E$ is replaced by a nonempty $w(E,E^*)$--compact convex subset of $E$.   We refer the reader to the introduction to \Sec{FUZZsec} for more precise details of these results.   These two fuzzy criteria are applied in \Thm{STRONGthm} to prove that every closed, monotone quasidense multifunction is {\em strongly maximal} in the sense of \cite[Theorems~6.1-2, pp.\ 1386--1387]{CONTROLLED}.
\par
In \Sec{FITZCHARsec}, we give sequential characterizations of $S^\F$ (see \Thm{SEQCHARthm}).   In view of Appendix 1, \Sec{APPsec}, \Thm{SEQCHARthm} actually provides sequential characterizations of the Gossez extension of $S$.
\par
In \Sec{FPsec}, we prove that a maximally monotone multifunction is quasidense if, and only if, it is of {\em type (FP)}.   (Type (FP) was introduced by Fitzpatrick--Phelps in \cite[Section~3(1992)]{FPONE} under the name of ``locally maximal monotone''.)
\par
In Appendix 2, \Sec{FPALTsec}, we indicate how the results of \Sec{FPsec} can be obtained without appealing to the results of \Sec{SUMSsec}, but using instead two results of Rockafellar.   
\par
There are many classes of maximally monotone multifunctions that we will not discuss in great detail in this paper, they share the common feature that they all require the bidual, $E\dbs$, of $E$ for their definition:  {\em Type (D)} and {\em dense type} were introduced by Gossez \big(see Gossez, \cite[Lemme 2.1, p.\ 375(1971)]{GOSSEZ} and Phelps,\break \cite[Section~3(1997)]{PRAGUE} for an exposition\big).   {\em Type (NI)} was first defined in\break \cite[Definition~10, p.\ 183(1996)]{RANGE}, and {\em type (ED)} was introduced in \cite[(1998)]{MANDM},\break (under the name of {\em type (DS)}\,).   Because of the work of Voisei and Z\u alinescu, \cite{VZ}, Marques Alves and Svaiter, \cite[Theorem~4.4, pp.\ 1084--1085(2010)]{ASD},\break Simons, \cite[Theorem~9.9(a), pp.\ 254--255(2011)]{SSDMON} and Bauschke, Borwein, Wang and Yao, \cite[Theorem~3.1, pp.\ 1878--1879(2012)]{BBWYFP} we now know that type (D), dense type, type (FP), type (NI) and type (ED) are all equivalent.   In this paper, we will only discuss in detail classes of the maximally monotone multifunctions, like quasidensity, that can be defined {\em without} reference to the bidual.
\par
The bidual is not mentioned explicitly in the {\em statements} of \Thms{STDthm}, \ref{FPVthm}, \ref{FUZZDthm}, \ref{FUZZEthm}, \ref{STRONGthm} and \ref{FPthm}\big((a)$\ifff$(b)\big), but our {\em proofs} of all of these results\break ultimately depend on the bidual at one point or another.   This raises the\break fascinating question whether there are proofs of any of these results that do not depend on the bidual.  This seems quite a challenge, because it would\break require finding substitutes for \Thms{STDthm} and \ref{STRthm}.
\par
The author would like to thank Mircea Voisei and Regina Burachik for some very helpful comments on an earlier version of this paper.
\section{Basic notation and definitions}\label{BANsec}
If $X$ is a nonzero real Banach space and $f\colon\ X \to \rbar$, we write $\dom\,f$ for the set $\big\{x \in X\colon\ f(x) \in \RR\big\}$.   $\dom\,f$ is the {\em effective domain} of $f$.   We say that $f$ is {\em proper} if $\dom\,f \ne \emptyset$.   We write $\PCLSC(X)$ for the set of all proper convex lower semicontinuous functions from $X$ into $\rbar$.   We write $X^*$ for the dual space of $X$ \big(with the pairing $\bra\cdot\cdot\colon X \times X^* \to \RR$\big).  If $f \in \PCLSC(X)$ then, as usual, we define the {\em Fenchel conjugate}, $f^*$, of $f$ to be the function on $X^*$ given by $x^* \mapsto \supn_X\big[x^* - f\big]$.
We write $X\dbs$ for the bidual of $X$ \big(with the pairing $\bra\cdot\cdot\colon X^* \times X\dbs \to \RR$\big).   If $f \in \PCLSC(X)$ and $f^* \in \PCLSC(X^*)$, we define $f\dbs\colon X\dbs \to \rbar$ by $f\dbs(x\dbs) := \sup_{X^*}\big[x\dbs - f^*\big]$.   If $x \in X$, we write $\wh x$ for the canonical image of $x$ in $X\dbs$, that is to say, for all $(x,x^*) \in X \times X^*$,\break $
\bra{x^*}{\wh x} = \bra{x}{x^*}$.\quad If $f \in \PCLSC(X)$ then the {\em subdifferential of} $f$ is the multifunction $\partial f\colon\ E  \toto E^*$ that satisfies $x^* \in \partial f(x) \ifff f(x) + f^*(x^*) = \bra{x}{x^*}$.   We write $X_1$ for the closed unit ball of $X$.
\medbreak
We now collect the definitions and results from \cite{PARTONE} that we shall use.   We suppose that $E$ is a nonzero real Banach space with dual $E^*$.   For all $(x,x^*) \in E \times E^*$, we write $\|(x,x^*)\| := \sqrt{\|x\|^2 + \|x^*\|^2}$.   We represent $(E \times E^*)^*$ by $E^* \times E\dbs$, under the pairing
$$\Bra{(x,x^*)}{(y^*,y\dbs)} := \bra{x}{y^*} + \bra{x^*}{y\dbs}.$$
The dual norm on $E^* \times E\dbs$ is given by  $\|(y^*,y\dbs)\| := \sqrt{\|y^*\|^2 + \|y\dbs\|^2}$.
\smallbreak
Now let $S\colon\ E \toto E^*$.   We write $G(S)$ for the graph of $S$, $D(S)$ for the domain of $S$ and $R(S)$ for the range of $S$.
We will always suppose that $G(S) \ne \emptyset$ (equivalently, $D(S) \ne \emptyset$ or $R(S) \ne \emptyset$).   We say that $S$ is {\em closed} if $G(S)$ is closed. 
\begin{definition}\label{QDdef}
We say that $S$ is {\em quasidense} if, for all $(x,x^*) \in E \times E^*$,
\begin{equation*}
\infn_{(s,s^*) \in G(S)}\big[\half\|s - x\|^2 + \half\|s^* - x^*\|^2 + \bra{s - x}{s^* - x^*}\big] \le 0.
\end{equation*}
See \cite[Example 7.1, eqn.\ (28), p.\ 1031]{PARTONE}.   {\em Quasidensity} is actually a special case of the more general concept of {\em $r_L$--density} considered in \cite[Section 4]{PARTONE}.
\end{definition} 
\begin{theorem}[Quasidensity and maximality]\label{RLMAXthm}
Let $S\colon\ E \toto E^*$ be closed, monotone and quasidense.   Then $S$ is maximally monotone. 
\end{theorem}
\begin{proof}
Let $(x,x^*) \in E \times E^*$, $G(S) \cup \{(x,x^*)\}$ be monotone and $\eps > 0$.   By hypothesis, there exists $(s,s^*) \in G(S)$ such that
\begin{equation*}
\half\|s - x\|^2 + \half\|s^* - x^*\|^2 + \bra{s - x}{s^* - x^*} < \eps.
\end{equation*}
Since $G(S) \cup \{(x,x^*)\}$ is monotone, $\bra{s - x}{s^* - x^*} \ge 0$.   Consequently,\break $\half\|s - x\|^2 + \half\|s^* - x^*\|^2 < \eps$.   However, $G(S)$ is closed:   thus, letting $\eps \to 0$,  $(x,x^*) \in G(S)$.   This completes the proof of the maximality of $S$.   \big(This proof is adapted from that of \cite[Lemma 4.7, p.\ 1027]{PARTONE} --- the result appears explicitly in \cite[Theorem 7.4(a), pp.\ 1032--1033]{PARTONE}.\big)     
\end{proof}
\begin{example}[The tail operator]\label{TAILex}
Let $E = \ell_1$, and define the linear map $T\colon\ \ell_1 \to \ell_\infty = E^*$ by $(Tx)_n = \sum_{k \ge n} x_k$.   It is well known that $T$ is maximally monotone.   However, $T$ is not quasidense:  see \cite[Example 7.10, pp.\ 1034--1035]{PARTONE}. 
\end{example}
\begin{definition}\label{PHdef}We define $\varphi_S\colon\ E \times E^* \to \rbar$ by
\begin{equation*}
\varphi_S(x,x^*) := \supn_{(s,s^*) \in G(S)}\big[\bra{s}{x^*} + \bra{x}{s^*} - \bra{s}{s^*}\big].
\end{equation*}
\end{definition}
\begin{lemma}\label{PHISlem}
Let $S\colon\ E \toto E^*$ be closed, monotone and quasidense.   Then:
\begin{equation}\label{PHIS1}
\left.
\begin{gathered}
\all\ (x,x^*) \in E \times E^*, \varphi_S(x,x^*) \ge \bra{x}{x^*} \quad\hbox{and}\\
\big\{(x,x^*) \in E \times E^*\colon\ \varphi_S(x,x^*) = \bra{x}{x^*}\big\} = G(S).
\end{gathered}
\right\}
\end{equation}
\end{lemma}
\begin{proof}
See \cite[Lemma 7.7, eqn.\ (30), p.\ 1034]{PARTONE}.
\end{proof}
\par
The next result is somewhat subtler, and will be used in \Sec{FITZDEFsec}.
\begin{lemma}\label{PHIATlem}
Let $S\colon\ E \toto E^*$ be closed, monotone and quasidense.   Then:
\begin{gather*}
\all\ (x,x^*) \in E \times E^*,\ {\varphi_S}^*(x^*,\wh x) \ge \bra{x}{x^*}\quad\hbox{ and}\\
\big\{(x,x^*) \in E \times E^*\colon\ {\varphi_S}^*(x^*,\wh x) = \bra{x}{x^*}\big\} = G(S).
\end{gather*}
\end{lemma}
\begin{proof}
See \cite[Lemma 7.7, eqn.\ (33), p.\ 1034]{PARTONE}.  (${\varphi_S}^*(x^*,\wh x)$ is denoted by ${\varphi_S}^@(x,x^*)$ in \cite{PARTONE}.   See \cite[Definition 3.1]{PARTONE}.)
\end{proof}
\par
In \Thm{PHTHthm}, we show that the function ${\varphi_S}$ can be used to give a\break simple criterion for quasidensity.   We will give more criteria in \Thms{FUZZDthm}, \ref{FUZZEthm} and \ref{FPthm}, and \Rem{NIrem}.
\begin{theorem}\label{PHTHthm}
Let $S\colon\ E \toto E^*$ be maximally monotone.    Then $S$ is quasidense if, and only if, for all $(y^*,y\dbs) \in E^* \times E\dbs$, ${\varphi_S}^*(y^*,y\dbs) \ge \bra{y^*}{y\dbs}$. 
\end{theorem}
\begin{proof}
See \cite[Corollary 6.4, p.\ 1029--1030]{PARTONE} for a more general result.
\end{proof}
\begin{remark}\label{PHTHrem}
We note that \eqref{PHIS1} is true even if $S$ is merely maximally monotone.   (This was actually proved by Fitzpatrick in \cite{FITZ}.)   On the other hand, \Thm{PHTHthm} gives a criterion for quasidensity.  
\end{remark}
\par
\Thm{RTRthm} below is a very important result.   By virtue of \Thm{RLMAXthm}, it generalizes Rockafellar's result that subdifferentials are maximally monotone.   The proof of \Lem{VANlem} uses two basic results from convex analysis.   The first of these, \Lem{BRlem}, is the Br{\o}ndsted--Rockafellar theorem, which was first proved in \cite[p.\ 608]{BRON}.   There are many variations of this result:  we will use \Cor{BRcor}, which appeared explicitly in \cite[Theorem 3.3, p.\ 1380]{CONTROLLED} (if not before).   The second result from convex analysis that we will use appears in \Lem{SUBSUMlem}.   This follows from Rockafellar's formula for the subdifferential of a sum.   See\break \cite[Theorem~3(a), pp.\ 85--86]{FENCHEL}.
\begin{lemma}\label{BRlem}
Let $h \in \PCLSC(E)$, $\inf_Eh > -\infty$, $\alpha,\beta > 0$, $u \in \dom\, h$ and $h(u) < \inf_Eh + \alpha\beta$.   Then there exists $(s,x^*) \in G(\partial h)$ such that $h(s) \le h(u)$, $\|s - u\| \le \alpha$ and $\|x^*\| \le \beta$.
\end{lemma}
\begin{corollary}\label{BRcor}
Let $h \in \PCLSC(E)$, $\inf_Eh > -\infty$ and $\beta > 0$.   Then there exists $(s,x^*) \in G(\partial h)$ such that $h(s) \le \inf_Eh + \beta$ and $\|x^*\| \le \beta$.
\end{corollary}
\begin{proof}
We can choose $u \in E$ such that $h(u) \le \inf_Eh + \beta$.   The result follows from \Lem{BRlem} with $\alpha = 1$.  
\end{proof}
\begin{remark}\label{BRrem}
We note that \Cor{BRcor} can be put in the following form:  {\em Let $h \in \PCLSC(E)$ and $\inf_Eh > -\infty$.   Then there exists a sequence $\{(s_n,x_n^*)\}$ of element of $G(\partial h)$ such that $\big(h(s_n,x_n^*),x_n^*\big) \to (\inf_Eh,0)$.}
\end{remark}
\begin{lemma}\label{SUBSUMlem}
Let $g\colon E \to \rbar$ be proper and convex and $k\colon E \to \RR$ be convex and continuous.   Then, for all $x \in E$, $\partial(g + k)(x) = \partial g(x) + \partial k(x)$.
\end{lemma}
\begin{lemma}\label{VANlem}
Let $g \in \PCLSC(E)$ and $\eps > 0$. Then there exists $(s,s^*) \in G(\partial g)$ such that $\half \|s\|^2+\bra{s}{s^*} + \half\|s^*\|^2 < \eps$.
\end{lemma}
\begin{proof} It is well known (from a separation theorem in $E \times E^*$) that $g$ dominates a continuous affine function.   Thus there exist $\gamma_0,\delta_0 \in \RR$ such that, for all $x \in E$, $g(x) \ge - \gamma_0\|x\| - \delta_0$.   Let $j(x) := \half\|x\|^2$.   Then, for all $x\in E$,
\begin{equation}\label{VAN1}
(g + j)(x) \ge \half\|x\|^2 - \gamma_0\|x\| - \delta_0 \ge \minn_{\lambda\in\RR}\big[\half\lambda^2 - \gamma_0\lambda - \delta_0\big] \in \RR. 
\end{equation}
Let $m:=\inf_E(g + j) > -\infty$.   By completing the square, there exists $M \in [\,0,\infty[\,$ such that
\begin{equation}\label{VAN2}
\half\lambda^2 - \gamma_0\lambda - \delta_0 \le m + 1 \qlr \lambda \le M.
\end{equation}
Choose $\beta \in \,]0,1]$ such that $2M\beta + \half\beta^2 < \eps$. From \Cor{BRcor}, there exists $(s,x^*) \in G(\partial(g + j))$ such $(g + j)(s) \le m + \beta \le m + 1$, and $\|x^*\|\le \beta$. \eqref{VAN1} implies that $\half\|s\|^2 - \gamma_0\|s\| - \delta_0 \le m + 1$, and so \eqref{VAN2} gives
\begin{equation}\label{VAN4}
\|s\| \le M.
\end{equation}
From \Lem{SUBSUMlem}, $\partial (g + j)(s) = \partial g(s) + \partial j(s)$, so there exists $s^*\in\partial g(s)$ such that $x^* - s^* \in \partial j(s)$, from which 
\begin{equation*}
\half\|s\|^2 = \bra{s}{x^* - s^*} - \half\|s^* - x^*\|^2 \quand  \|s^* - x^*\| = \|s\|.
\end{equation*}
Thus, since $\|x^*\|\le \beta$, $\|s^*\| \le \|s\| + \beta$, and so   
\begin{align*}
\half\|s\|^2 + \bra{s}{s^*} + \half\|s^*\|^2
&= \bra{s}{x^*} - \half\|s^* - x^*\|^2 + \half\|s^*\|^2\\
&\le \|s\|\beta - \half\|s\|^2 + \half(\|s\| + \beta)^2 \le 2\|s\|\beta + \half\beta^2.
\end{align*}
From \eqref{VAN4}, $\half\|s\|^2 + \bra{s}{s^*} + \half\|s^*\|^2 \le 2M\beta + \half\beta^2 < \eps$.
\end{proof}
\begin{theorem}\label{RTRthm}
Let $f \in \PCLSC(E)$.   Then  $\partial f$ is closed, monotone and\break quasidense.
\end{theorem}
\begin{proof}
Let $(x,x^*) \in E \times E^*$.   We apply \Lem{VANlem} to the function\break $g := f(\cdot + x) - x^*$, and the result follows since $G(\partial g) = G(\partial f) - (x,x^*)$.
\end{proof} 
\begin{remark}\label{RTRrem}
Another proof of \Thm{RTRthm} can be found in \cite[Theorem 7.5, p.\ 1033]{PARTONE}, using \cite[Theorem 5.2]{PARTONE} and \cite[Corollary 4.5]{PARTONE}.   The proof given here is a simplified version of that given in \cite[Theorem 8.4]{V1}.   This result was extended to nonconvex functions in \cite[Theorem 3.2, pp.\ 634--635]{SW} (using an appropriate definition of subdifferential for a nonconvex function).  
\end{remark}
We now give two simple but significant applications of \Thm{RTRthm}. 
\begin{example}\label{DNORMALex}
Let $\KT$ be a nonempty $w(E^*,E)$--compact convex subset of $E^*$.   We define the continuous sublinear functional $\tauKT$ on $E$ by $\tauKT := \max\Bra{\cdot}{\KT}$.   From \Thm{RTRthm},
\begin{equation}\label{DNORMAL2}
\ptKT\hbox{ is closed, monotone and quasidense.}
\end{equation}
By direct computation, $\tauKT^* = \IKT$, where $\IKT(x^*) = 0$ if $x^* \in \KT$ and $\IKT(x^*) = \infty$ if $x^* \in E^* \setminus \KT$, and so
\begin{equation}\label{DNORMAL1}
\left.
\begin{aligned}
x^* \in \ptKT(x)
&\iff \tauKT(x) + \IKT(x^*) = \bra{x}{x^*}\\
&\iff x^* \in \KT\hbox{ and }\bra{x}{x^*} = \max\Bra{x}{\KT}.
\end{aligned}
\right\}
\end{equation}
Using the $w(E^*,E)$--compactness of $\KT$, this implies that $D(\ptKT) = E$ and $R(\ptKT) \subset \KT$.   On the other hand, if $x^* \in \KT$ then $x^* \in \ptKT(0)$. To sum up:
\begin{equation}\label{DNORMAL3}
D(\ptKT) = E\quad\hbox{and}\quad R(\ptKT) = \KT.  
\end{equation}
\end{example}
\begin{example}\label{NORMALex}
Let $K$ be a nonempty bounded closed convex subset of $E$ and $\I_K(x) = 0$ if $x \in K$ and $\I_K(x) = \infty$ if $x \in E \setminus K$.   From \Thm{RTRthm},
\begin{equation}\label{NORMAL3}
\pIK\hbox{ is closed, monotone and quasidense.}
\end{equation}
We define the continuous, sublinear functional $\sigma_K$ on $E^*$ by $\sigma_K := \sup\bra{K}{\cdot}$.   By direct computation, ${\I_K}^* = \sigma_K$, and so
\begin{equation}\label{NORMAL1}
\left.\begin{aligned}
x^* \in \pIK(x)
&\iff \I_K(x) + \sigma_K(x^*) = \bra{x}{x^*}\\
&\iff x \in K\ \hbox{and}\ \bra{x}{x^*} = \sup\bra{K}{x^*}.
\end{aligned}\right\}
\end{equation}
($\pIK$ is the {\em normal cone multifunction} of $K$.)   \eqref{NORMAL1} clearly implies that\break $D(\pIK) \subset K$.   On the other hand, if $x \in K$ then $0 \in \pIK(x)$.   Consequently,
\begin{equation}\label{NORMAL2}
D(\pIK) = K.
\end{equation}
Now suppose, in addition, that $K$ is $w(E,E^*)$--compact.   From \eqref{NORMAL1},
\begin{equation}\label{NORMAL5}
R(\pIK) = E^*.
\end{equation}
Let $\wh K := \{\wh x\colon x \in K\}$. It is easy to see that   
$\wh K$ is $w(E\dbs,E^*)$--closed, from which
\begin{equation}\label{NORMAL4}
{\sigma_K}^* = \I_{\wh K}.
\end{equation}
\end{example}
\section{The Fitzpatrick extension}\label{FITZDEFsec}
\begin{definition}[The Fitzpatrick extension]\label{FITZdef}  (See \cite[Definition 8.5, p.\ 1037]{PARTONE}.)   Let the notation be as in \Sec{BANsec} and $S\colon\ E \toto E^*$ be closed, monotone and quasidense.  From \Thm{PHTHthm}, for all $(y^*,y\dbs) \in E^* \times E\dbs$, ${\varphi_S}^*(y^*,y\dbs) \ge \bra{y^*}{y\dbs}$. 
We define $S^\F\colon\ E^* \toto E\dbs$ so that
\begin{equation}\label{FITZ1}
(y^*,y\dbs) \in G(S^\F) \hbox{ exactly when } {\varphi_S}^*(y^*,y\dbs) = \bra{y^*}{y\dbs}.
\end{equation}
It is easily seen that $S^\F$ is monotone.  Then, from \Lem{PHIATlem},
\begin{equation}\label{FITZ2}
(x,x^*) \in G(S) \iff (x^*,\wh x) \in G(S^\F),
\end{equation}
and so $S^\F$ is, in some sense, an extension of $S$.   We will describe $S^\F$ as the {\em Fitzpatrick extension of $S$}.   (We note that $\Phi_{G(S)} = \varphi_S$ in \cite{PARTONE}.)   
\end{definition}
\par
Following the notation introduced in \cite[Example 7.1, p.\ 1031]{PARTONE}, we define the map $L\colon\ E \times E^* \to E^* \times E\dbs$ by $L(x,x^*) := (x^*,\wh{x})$.
\begin{lemma}\label{HATlem}
Let $S\colon\ E \toto E^*$ be closed, monotone and quasidense and\break $R(S^\F) \subset \wh E := \{\wh x\colon x \in E\}$.   Then $G(S^\F) = L\big(G(S)\big)$.
\end{lemma}
\begin{proof}
If $(y^*,y\dbs) \in G(S^\F)$ then $y\dbs \in R(S^\F) \subset \wh E$, and so there exists $x \in E$ such that $\wh x = y\dbs$.   Thus $(y^*,\wh x) \in G(S^\F)$.   From \eqref{FITZ2}, $(x,y^*) \in G(S)$.   Since $(y^*,y\dbs) = (y^*,\wh x) = L(x,y^*)$, this establishes that   $G(S^\F) \subset L\big(G(S)\big)$.   On the other hand, if $(y^*,y\dbs) \in L\big(G(S)\big)$ then there exists $(x,x^*) \in G(S)$ such that $(y^*,y\dbs) = (x^*,\wh x)$.   From \eqref{FITZ2}, $(x^*,\wh x) \in G(S^\F)$, that is to say $(y^*,y\dbs) \in G(S^\F)$.   Thus we have proved that $L\big(G(S)\big) \subset G(S^\F)$.
\end{proof}
\begin{lemma}\label{FITZDIFFlem}
Let $f \in \PCLSC(E)$.   Then $G\big((\partial f)^\F\big) \subset G\big(\partial(f^*)\big)$.
\end{lemma}
\begin{proof}
Let $(x,x^*) \in E \times E^*$.   Then, from \Def{PHdef}, the definition of $\partial f$ and the Fenchel--Young inequality,
\begin{align*}
\varphi_{\partial f}(x,x^*)
&= \supn_{(s,s^*) \in G(\partial f)}\big[\bra{x}{s^*} + \bra{s}{x^*} - \bra{s}{s^*}\big]\\
&= \supn_{(s,s^*) \in G(\partial f)}\big[\bra{x}{s^*} + \bra{s}{x^*} - f(s) - f^*(s^*)\big]\\
&\le \supn_{s^* \in E^*}\big[\bra{x}{s^*} - f^*(s^*)] +  \supn_{s \in E}\big[\bra{s}{x^*} - f(s)\big]\\
&\le f(x) + f^*(x^*).
\end{align*}
Consequently, for all $(y^*,y\dbs) \in E^* \times E\dbs$,
\begin{align*}
{\varphi_{\partial f}}^*(y^*,y\dbs)
&= \supn_{(x,x^*) \in E \times E^*}\big[\bra{x}{y^*} + \bra{x^*}{y\dbs} - \varphi_{\partial f}(x,x^*)\big]\\
&\ge \supn_{(x,x^*) \in E \times E^*}\big[\bra{x}{y^*} + \bra{x^*}{y\dbs} - f(x) - f^*(x^*)\big]\\
&= \supn_{x \in E}\big[\bra{x}{y^*} - f(x)\big] + \supn_{x^* \in E^*}\big[\bra{x^*}{y\dbs} - f^*(x^*)\big]\\
&= f^*(y^*) + f\dbs(y\dbs) \ge \bra{y^*}{y\dbs}.
\end{align*}
The result now follows from \eqref{FITZ1}.  
\end{proof}
\begin{lemma}\label{FITZPARTAUlem}
Let $\KT$ be a nonempty $w(E^*,E)$--compact convex subset of $E^*$ and $y\dbs \in {\partial\tauKT}^\F(y^*)$.   Then $y^* \in \KT\hbox{ and }\bra{y^*}{y\dbs} = \sup\Bra{\KT}{y\dbs}$.
\end{lemma}
\begin{proof}
As we observed in \Ex{DNORMALex}, ${\tauKT}^* = \I_\KT$ and so, from \Lem{FITZDIFFlem}, $y\dbs \in \partial\I_\KT(y^*)$, that is to say $\I_\KT(y^*) + {\I_\KT}^*(y\dbs) = \bra{y^*}{y\dbs}$.
This gives the desired result.
\end{proof}
\begin{lemma}\label{PHINKlem}
Let $K$ be a nonempty $w(E,E^*)$--compact convex subset of $E$.
\smallbreak\noindent
{\rm(a)}\enspace Let $y\dbs \in {\pIK}^\F(y^*)$.   Then $y\dbs \in \wh K\hbox{ and }\bra{y^*}{y\dbs} = \sup\bra{K}{y^*}$.
\smallbreak\noindent
{\rm(b)}\enspace $R({\pIK}^\F) \subset \wh K \subset \wh E$.
\end{lemma}
\begin{proof}
(a)\enspace As we observed in \Ex{NORMALex}, ${\I_K}^* = \sigma_K$ and so, from \Lem{FITZDIFFlem}, $y\dbs \in \partial\sigma_K(y^*)$, that is to say, $\sigma_K(y^*) + {\sigma_K}^*(y\dbs) = \bra{y^*}{y\dbs}$.   From \eqref{NORMAL4}, $\sigma_K(y^*) + \I_{\wh K}(y\dbs) = \bra{y^*}{y\dbs}$, which gives (a).   (b) is immediate from (a).   
\end{proof}
\begin{theorem}\label{AFMAXthm}
Let $S\colon\ E \toto E^*$ be closed, monotone and quasidense.   Then $S^\F\colon E^* \toto E\dbs$ is maximally monotone.
\end{theorem}
\begin{proof}
See \cite[Lemma 12.5, p.\ 1047]{PARTONE}.   There is also a sketch of a proof in Appendix 1, \Sec{APPsec}.
\end{proof}
\par
We end this section by calculating the Fitzpatrick extension of a general subdifferential, as well as computing the Fitzpatrick extensions of the two closed, monotone, quasidense multifunctions introduced in \Exs{DNORMALex} and \ref{NORMALex}. 
\begin{theorem}\label{FITZDIFFthm}
Let $f \in \PCLSC(E)$.   Then $(\partial f)^\F = \partial(f^*)$.
\end{theorem}
\begin{proof}
This is immediate from \Lem{FITZDIFFlem} and \Thm{AFMAXthm}.  
\end{proof}
\begin{theorem}[The Fitzpatrick extension of $\partial\tauKT$]\label{FITZPARTAUthm}
Let $\KT$ be a nonempty $w(E^*,E)$--compact convex subset of $E^*$.   Then
\begin{equation*}
(y^*,y\dbs) \in G\big({\partial\tauKT}^\F\big) \iff y^* \in \KT\hbox{ and }\bra{y^*}{y\dbs} = \sup\Bra{\KT}{y\dbs}.
\end{equation*}
\end{theorem}
\begin{proof}
As we observed in \Ex{DNORMALex}, ${\tauKT}^* = \I_\KT$ and so, from \Thm{FITZDIFFthm}, ${\partial\tauKT}^\F = \partial\I_\KT$.   The result now follows by using the technique of \Lem{FITZPARTAUlem}. 
\end{proof}
\begin{theorem}[The Fitzpatrick extension of $\partial\I_K$]\label{PHINKthm}
Let $K$ be a nonempty $w(E,E^*)$--compact convex subset of $E$.   Then
\begin{equation*}
(y^*,y\dbs)\in G\big({\pIK}^\F\big) \iff y\dbs \in \wh K\hbox{ and }\bra{y^*}{y\dbs} = \sup\bra{K}{y^*}.
\end{equation*}
\end{theorem}
\begin{proof}
As we observed in \Ex{NORMALex}, ${\I_K}^* = \sigma_K$ and so, from \Thm{FITZDIFFthm}, ${\partial\I_K}^\F = \partial\sigma_K$.   The result now follows by using the technique of \Lem{PHINKlem}.\end{proof}
%
\begin{remark}\label{FGrem}
\par
\cite[Theorem 12.4(a), p.\ 1047]{PARTONE} implies that if $S\colon\ E \toto E^*$ is closed, monotone and quasidense and $(y^*,y\dbs) \in E^* \times E\dbs$ then
\begin{equation*}
(y^*,y\dbs) \in G(S^\F) \iff \infn_{(s,s^*) \in G(S)}\bra{s^* - y^*}{\wh s - y\dbs} = 0.
\end{equation*}
This is equivalent to the result proved in \eqref{NI3}.   There are more characterizations of $S^\F$ in \Thm{SEQCHARthm} and \eqref{NI4}.   We do not know if $S^\F\colon E^* \toto E\dbs$ is necessarily quasidense in the context of \Thm{AFMAXthm}, but \Thms{FITZDIFFthm} and \ref{RTRthm} show that {\em if $f \in \PCLSC(E)$ then $(\partial f)^\F$ is quasidense}.   \Thm{FITZDIFFthm} is equivalent to\break \cite[Th\'eor\`eme~3.1, pp.\ 376--378]{GOSSEZ}.   
\par
\Lem{FITZDIFFlem} is the ``easy half'' of \Thm{FITZDIFFthm}.   \Thm{FITZDIFFthm} and its two consequences, \Thms{FITZPARTAUthm} and \ref{PHINKthm}, will not be used any more in this paper, while \Lem{FITZPARTAUlem} will be used in \Thm{FPthm}, and \Lem{PHINKlem} will be used in \Thm{FUZZEthm}.   We have separated the proof of \Thm{FITZDIFFthm} into these two parts so that the reader is not obliged to wade through the complexities of the later part of \cite{PARTONE} or Appendix 1, \Sec{APPsec}, to understand the logic of the rest of this paper.   This raises the issue of finding a simple, direct proof of the inclusion $G\big(\partial(f^*)\big) \subset G\big((\partial f)^\F\big)$ in the context of \Thm{FITZDIFFthm}.   To date, we have not found such a proof.    
\end{remark}
\section{Two sum theorems}\label{SUMSsec}
The proofs of the two results in this section use, among other things, the\break bivariate version of the Attouch-Brezis theorem first proved in Simons--Z\u{a}linescu \cite[Section~4, pp.\ 8--10]{SZNZ}.   \Thm{STDthm} will be used in  \Thms{FPVthm} and \ref{FUZZDthm}, while \Thm{STRthm} will be used in  \Thms{FUZZEthm} and \ref{FPthm}.    
\begin{theorem}[Sum theorem with domain constraints]\label{STDthm}
Let $S,T\colon\ E \toto E^*$ be closed, monotone and quasidense and either $D(S) \cap \intr\,D(T) \ne \emptyset$ or\break $\intr\,D(S) \cap D(T) \ne \emptyset$.   Then $S + T$ is closed, monotone and quasidense.
\end{theorem}
\begin{proof}
See \cite[Theorem 8.4(a)$\lr$(d), pp.\ 1036--1037]{PARTONE}.
\end{proof}
\begin{theorem}[Sum theorem with range constraints]\label{STRthm}
Let $S,T\colon\ E \toto E^*$ be closed, monotone and quasidense and either $R(S) \cap \intr\,R(T) \ne \emptyset$ or\break $\intr\,R(S) \cap R(T) \ne \emptyset$. Then the multifunction $y \mapsto (S^\F + T^\F)^{-1}(\wh y)$ is closed, monotone and quasidense.
If, further, $R(T^\F) \subset \wh E$, then the {\em parallel sum} $S \parallel T := (S^{-1} + T^{-1})^{-1}$ is closed, monotone and  quasidense.
\end{theorem}
\begin{proof}
The first observation was established in \cite[Theorem 8.8, p.\ 1039]{PARTONE}.   If $R(T^\F) \subset \wh E$ then, from \Lem{HATlem}, $G(T^\F) = L\big(G(T)\big)$.   Now in \cite[Definition 8.5, eqn.\ (46), p.\ 1037]{PARTONE}, what was actually defined was the Fitzpatrick extension of a {\em subset} of $E \times E^*$ rather than that of a {\em multifunction from $E$ into $E^*$}, and the relation between them is $G(T^\F) = G(T)^\F$.   Thus $G(T)^\F = L\big(G(T)\big)$, and it now follows from \cite[Theorem 8.8]{PARTONE} that $S \parallel T$ is closed, monotone and quasidense.
\end{proof}
%
\section{Type (FPV)}\label{FPVsec}
\begin{definition}\label{FPVdef}
Let $S\colon\ E \toto E^*$ be monotone.   We say that $S$ is {\em of type (FPV)} or {\em maximally monotone locally} if whenever $U$ is an open convex subset of $E$, $U \cap D(S) \neq \emptyset$, $(w,w^*) \in U \times E^*$ and
\begin{equation}\label{FPV1}
(s,s^*) \in G(S)\quand s \in U \qlr \bra{s - w}{s^* - w^*}\ge 0,
\end{equation}
then $(w,w^*) \in G(S)$.   (If we take $U = E$, we see that every monotone\break multifunction of type (FPV) is maximally monotone.)
\end{definition}
\begin{theorem}\label{FPVthm}
Let $S\colon\ E \toto E^*$ be closed, monotone and quasidense.      Then $S$ is maximally monotone of type (FPV).
\end{theorem}
\begin{proof}
Let $U$ be an open convex subset of $E$, $U \cap D(S) \neq \emptyset$, $(w,w^*) \in U \times E^*$ and \eqref{FPV1} be satisfied.   Let $y \in U \cap D(S)$.   Since the segment $[w,y]$ is a compact subset of the open set $U$, we can choose $\eps > 0$ so that $K := [w,y] + \eps E_1 \subset U$.   $K$ is bounded, closed and convex.   Let $T := \pIK$.   From \eqref{NORMAL2},  $D(S) \cap \intr\,D(T) = D(S) \cap \intr\,K \ni y$, and \eqref{NORMAL3} and \Thm{STDthm} imply that $S + T$ is closed and quasidense.   Let $\eta > 0$.   Then there exists $(s,u^*) \in G(S + T)$ such that 
\begin{equation}\label{FPV3}
\half\|s - w\|^2 + \half\|u^* - w^*\|^2 + \bra{s - w}{u^* - w^*} < \eta.
\end{equation}
We can choose $s^* \in S(s)$ and $x^* \in T(s)$ such that $s^* + x^* = u^*$.   Since\break $s \in D(T) = K \subset U$, \eqref{FPV1} implies that $\bra{s - w}{s^* - w^*}\ge 0$ and, since $w \in K$, \eqref{NORMAL1} implies that $\bra{s}{x^*} = \sup\bra{K}{x^*} \ge \bra{w}{x^*}$.   Combining together these two inequalities, we have $\bra{s - w}{u^* - w^*} = \bra{s - w}{s^* - w^*} + \bra{s}{x^*} - \bra{w}{x^*} \ge 0$.   From \eqref{FPV3},
\begin{equation*}
\half\|s - w\|^2 + \half\|u^* - w^*\|^2 < \eta.
\end{equation*}
Thus, taking $\eta$ arbitrarily small and using the fact that $S + T$ is closed,\break $w^* \in (S + T)(w)$, from which there exist $s_0^* \in S(w)$ and $x_0^* \in T(w)$ such that $s_0^* + x_0^* = w^*$. From \eqref{NORMAL1}, $\bra{w}{x_0^*} = \sup\bra{K}{x_0^*}$.  Since $w \in \intr\,K$, $x_0^* = 0$, from which $s_0^* = w^*$.   Thus $(w,w^*) = (w,s_0^*) \in G(S)$, as required. 
\end{proof}
\begin{remark}
We do not know of an example of a maximally monotone\break multifunction that is not of type (FPV).      The tail operator (see \Ex{TAILex}) does not provide an example because it was proved in  Fitzpatrick--Phelps,\break \cite[Theorem 3.10, p.\ 68]{FITZTWO} that if $S\colon\ E \toto E^*$ is maximally monotone and $D(S) = E$ then $S$ is of type (FPV).   This question is closely related to the sum problem.   \big(See \cite[Theorem~44.1, p. 170]{HBM}.\big)   \Thm{FPVthm} can also be deduced from Voisei--Z\u alinescu, \cite[Remark 3.6, p.\ 1024]{VZ}.
\end{remark}
\section{Fuzzy criteria for quasidensity}\label{FUZZsec}
Let $S\colon\ E \toto E^*$ be closed and monotone.   From \Def{QDdef}, $S$ is quasidense if, and only if, for all $(w,w^*) \in E \times E^*$ and $\eta > 0$, there exists $(s,s^*) \in G(S)$ such that $\half\|s - w\|^2 + \half\|s^* - w^*\|^2 + \bra{s - w}{s^* - w^*} < \eta$.   In \Thm{FUZZDthm}, we show that this is equivalent to a formally much stronger condition in which $w^*$ is replaced by any nonempty $w(E^*,E)$--compact convex subset of $E^*$ and, in \Thm{FUZZEthm}, we show that this is equivalent to a formally much stronger condition in which $w$ is replaced by any nonempty $w(E,E^*)$--compact convex subset of $E$. \Thm{FUZZDthm} and \Thm{FUZZEthm} lead to \Thm{STRONGthm}, in which we prove that every closed, monotone quasidense multifunction is {\em strongly maximal}.
\begin{theorem}[A criterion for quasidensity in which $w^*$ becomes fuzzy]\label{FUZZDthm}
Let $S\colon\ E \toto E^*$ be closed and monotone.   Then {\rm (a)}$\ifff${\rm (b)}.
\par\noindent
{\rm (a)}\enspace $S$ is quasidense.
\par\noindent
{\rm (b)}\enspace For all $w \in E$, nonempty $w(E^*,E)$--compact convex subsets $\WT$ of $E^*$ and $\eta > 0$, there exists $(s,s^*) \in G(S)$ such that
\begin{equation*}
\half\|s - w\|^2 + \half\dist(s^*,\WT)^2 + \max\Bra{s - w}{s^* - \WT} < \eta.
\end{equation*} 
\end{theorem}
\begin{proof}
\par
First suppose that (a) is true.  Let $w \in E$ and $\WT$ be a nonempty $w(E^*,E)$--compact convex subset of $E^*$.   We define the multifunction\break $_wS\colon E \toto E^*$ so that $G({_wS}) = G(S) - (w,0)$, and write $T:= \partial\tauWT$.   Clearly, $_wS$ is closed, monotone and quasidense and, from  \eqref{DNORMAL2} and \eqref{DNORMAL3} with $\KT := -\WT$, $T$ is closed,  monotone and quasidense and $D(T) = E$.   \Thm{STDthm} now implies that $_wS + T$ is also closed, monotone and quasidense.   Thus, for all $\eta > 0$, there exist $x \in E$, $s^* \in {_wS}(x)$ and $x^* \in T(x)$ such that
\begin{equation*}
\half\|x\|^2 + \half\|s^* + x^*\|^2 + \bra{x}{s^* + x^*} < \eta.
\end{equation*}
From \eqref{DNORMAL1}, $x^* \in -\WT$ and $\bra{x}{x^*} = \max\Bra{x}{-\WT}$, and so $\max\Bra{x}{s^* - \WT} = \bra{x}{s^* + x^*}$.   Since $s^* \in {_wS}(x)$, $s^* \in S(x + w)$.  Let $s := x + w$:  then $x = s - w$ and $(s,s^*) \in G(S)$.   (b) now follows since $\|s^* + x^*\| = \|s^* - (-x^*)\| \ge \dist(s^*,\WT)$.
\par
Suppose, conversely, that (b) is true.  Let $(w,w^*) \in E \times E^*$ and $\eta > 0$.   Define $\WT := \{w^*\}$, and let $(s,s^*)$  be as in (b).   Then (b) implies that
\begin{equation*}
\half\|s - w\|^2 + \half\|s^* - w^*\|^2 + \bra{s - w}{s^* - w^*} < \eta.
\end{equation*}
Thus $S$ is quasidense, and so (a) is true.
\end{proof}
\par
We can think of the next result as a ``dual'' to \Thm{FUZZDthm}.
\begin{theorem}[A criterion for quasidensity in which $w$ becomes fuzzy]\label{FUZZEthm}
Let $S\colon\ E \toto E^*$ be closed and monotone.   Then {\rm (a)$\ifff$(b)}.
\par\noindent
{\rm (a)}\enspace $S$ is quasidense.
\par\noindent
{\rm (b)}\enspace For all nonempty $w(E,E^*)$--compact convex subsets $W$ of $E$, $w^* \in E^*$ and $\eta > 0$, there exists $(s,s^*) \in G(S)$ such that
\begin{equation*}
\half\dist(s,W)^2 + \half\|s^* - w^*\|^2 + \max\bra{s - W}{s^* - w^*} < \eta.
\end{equation*}
\end{theorem}
\begin{proof}
First suppose that (a) is true.   Let $W$ be a nonempty $w(E,E^*)$--compact convex subset of $E$ and $w^* \in E^*$.   Clearly, $S - w^*$ is closed, monotone and quasidense.   Let $T := \partial\I_{-W}$.   From \eqref{NORMAL3} and \eqref{NORMAL5} with $K := -W$, $T$ is closed, monotone and quasidense and $R(T) = E^*$.   \Lem{PHINKlem}(b) and \Thm{STRthm} now imply that $(S - w^*) \parallel T$ is also closed, monotone and quasidense.   Thus, for all $\eta > 0$, there exist $x^* \in E^*$, $(s,x^*) \in G(S - w^*)$ and  $(x,x^*) \in G(T)$ such that
\begin{equation*}
\half\|s + x\|^2 + \half\|x^*\|^2 + \bra{s + x}{x^*} < \eta.
\end{equation*}
From \eqref{NORMAL1}, $x \in -W$ and $\bra{x}{x^*} = \max\bra{-W}{x^*}$, from which $\max\bra{s - W}{x^*} = \bra{s + x}{x^*}$.   Since $(s,x^*) \in G(S - w^*)$, $x^* + w^* \in S(s)$.   Let $s^*:= x^* + w^*$: then $x^* = s^* - w^*$ and $(s,s^*) \in G(S)$.   (b) now follows since $\|s + x\| = \|s - (-x)\| \ge \dist(s,W)$.
\par
Suppose, conversely, that (b) is true.  Let $(w,w^*) \in E \times E^*$ and $\eta > 0$.   Define $W := \{w\}$, and let $(s,s^*)$ be as in (b).   Then (b) implies that
\begin{equation*}
\half\|s - w\|^2 +   \half\|s^* - w^*\|^2 + \bra{s - w}{s^* - w^*} < \eta.
\end{equation*}
Thus $S$ is quasidense, and so (a) is true.
\end{proof}
\par
\Thm{STRONGthm} below states that a closed, monotone quasidense multifunction is {\em strongly maximal} in the sense of \cite[Theorems~6.1-2, pp.\ 1386--1387]{CONTROLLED}.   It is worth pointing out that we do not know of a maximally monotone multifunction that is not strongly maximal.   The tail operator (see \Ex{TAILex}) does not provide an example because it was proved in  Bauschke--Simons, \cite[Theorem 1.1, pp.\ 166--167]{BS} that if $S\colon\ D(S) \subset E \to E^*$ is linear and maximally monotone then $S$ is strongly maximal.
\par
We will use the computational rules contained in \Lems{NORMWKlem} and \ref{DISTlem} below.   ${\cal T}_E$ stands for the norm topology of $E$ and ${\cal T}_{E^*}$ for the norm topology of $E^*$.   We note that \Lem{NORMWKlem} is true even if $S$ is merely maximally monotone.   (See \Rem{PHTHrem}.)   
\begin{lemma}\label{NORMWKlem}
Let $S\colon\ E \toto E^*$ be closed, monotone and quasidense, $\{(s_\beta,s_\beta^*)\}$ be a {\em bounded} net of elements of $G(S)$ and $(z,z^*) \in E \times E^*$.
\par\noindent
{\rm (a)} If $(s_\beta,s_\beta^*) \to (z,z^*)$ in ${\cal T}_E\ \times w(E^*,E)$ then $(z,z^*) \in G(S)$.
\par\noindent
{\rm (b)} If $(s_\beta,s_\beta^*) \to (z,z^*)$ in $w(E,E^*) \times {\cal T}_{E^*}$ then $(z,z^*) \in G(S)$.
\end{lemma}
\begin{proof}
Note that   
$\bra{s_\beta}{s_\beta^*} - \bra{z}{z^*} = \bra{s_\beta - z}{z^*} + \bra{s_\beta - z}{s_\beta^* - z^*} + \bra{z}{s_\beta^* - z^*}$.   In both cases, $\lim_\beta\bra{s_\beta - z}{z^*} = 0$ and $\lim_\beta\bra{z}{s_\beta^* - z^*} = 0$.   In case (a), we have $\lim_\beta\|s_\beta - z\| = 0$ and $\sup_\beta\|s_\beta^* - z^*\| < \infty$. In case (b), we have\break $\sup_\beta\|s_\beta - z\| < \infty$ and $\lim_\beta\|s_\beta^* - z^*\| = 0$.   Thus, in both cases, we have\break $\lim_\beta\bra{s_\beta - z}{s_\beta^* - z^*} = 0$, and so $\lim_\beta\bra{s_\beta}{s_\beta^*} = \bra{z}{z^*}$.   It now follows from \eqref{PHIS1} that $\lim_\beta\varphi_S(s_\beta,s_\beta^*) = \bra{z}{z^*}$.  Since $\varphi_S$ is $w(E,E^*) \times w(E^*,E)$ lower semicontinuous and $(s_\beta,s_\beta^*) \to (z,z^*)$ in $w(E,E^*) \times w(E^*,E)$, it follows that $\varphi_S(z,z^*) \le \bra{z}{z}$, and another application of \eqref{PHIS1} gives $(z,z^*) \in G(S)$.    
\end{proof}
\begin{lemma}\label{DISTlem}
{\rm(a)}\enspace Let $\WT$ be a nonempty $w(E^*,E)$--compact subset of $E^*$ and $\{s_\alpha^*\}$ be a net of elements of $E^*$ such that $\lim_\alpha\dist(s_\alpha^*,\WT) = 0$.   Then there exist $w^* \in \WT$ and a subnet $\{s_\beta^*\}$ of $\{s_\alpha^*\}$ such that $s_\beta^* \to w^*$ in $w(E^*,E)$.
\par\noindent
{\rm(b)}\enspace Let $W$ be a nonempty $w(E,E^*)$--compact subset of $E$ and $\{s_\alpha^*\}$ be a net of elements of $E$ such that $\lim_\alpha\dist(s_\alpha,W) = 0$.   Then there exist $w \in W$ and a subnet $\{s_\beta\}$ of $\{s_\alpha\}$ such that $s_\beta \to w$ in $w(E,E^*)$.
\par\noindent
\end{lemma}
\begin{proof}
In case (a), there exists a net $\{w_\alpha^*\}$ of elements of $\WT$ such that\break $\lim_\alpha\|w_\alpha^* - s_\alpha^*\| = 0$.      Since $\WT$ is $w(E^*,E)$--compact, there exist $w^* \in \WT$ and a subnet $\{w_\beta^*\}$ of $\{w_\alpha^*\}$ such that $w_\beta^* \to w^*$ in $w(E^*
,E)$.   (a) follows since\break manifestly $s_\beta^* \to w^*$ in $w(E^*,E)$ also.   The proof of (b) is similar. 
\end{proof}
\begin{theorem}[Quasidensity implies ``strong maximality'']\label{STRONGthm}
Let $S\colon\ E \toto E^*$ be closed, monotone and quasidense.  Then, whenever $w \in E$ and $\WT$ is a nonempty $w(E^*,E)$--compact convex subset of $E^*$ such that, 
\begin{equation}\label{STRONG1}
\all\ (s,s^*) \in G(S),\ \max\bra{s - w}{s^* - \WT} \ge 0,\end{equation}
then $Sw \cap \WT \ne \emptyset$ and, further, whenever $W$ is a nonempty $w(E,E^*)$--compact convex subset of $E$, $w^* \in E^*$ and,
\begin{equation}\label{STRONG3}
\all\ (s,s^*) \in G(S),\ \max\bra{s - W}{s^* - w^*} \ge 0,
\end{equation}
then $w^* \in S(W)$.
\end{theorem}
\begin{proof}
Let $w \in E$, $\WT$ be a nonempty $w(E^*,E)$--compact convex subset of $E^*$, and \eqref{STRONG1} be satisfied.   From \Thm{FUZZDthm}, for all $\eta > 0$, there exists $(s,s^*) \in G(S)$ such that
$$\half\|s - w\|^2 + \half\dist(s^*,\WT)^2 < \eta.$$
Thus there exist a sequence $\{(s_n,s_n^*)\}_{n \ge 1}$ of elements of $G(S)$ such that\break $\lim_n\|s_n - w\| = 0$ and $\lim_n\dist(s_n^*,\WT) = 0$.  From \Lem{DISTlem}(a), there\break exist $w^* \in \WT$ and a subnet $\{(s_\beta,s_\beta^*)\}$ of $\{(s_n,s_n^*)\}$ such that $s_\beta^* \to w^*$ in $w(E^*,E)$.   Obviously, $s_\beta \to w$ in ${\cal T}_E$ and so, from \Lem{NORMWKlem}(a), $(w,w^*) \in G(S)$.  So $Sw \cap \WT \ne \emptyset$, as required.    
\par
Now let $W$ be a nonempty $w(E,E^*)$--compact convex subset of $E$, $w^* \in E^*$ and \eqref{STRONG3} be satisfied.   From \Thm{FUZZEthm}, for all $\eta > 0$, there exists $(s,s^*) \in G(S)$ such that
$$\half\dist(s,W)^2 + \half\|s^* - w^*\|^2 < \eta.$$
Thus there exists a sequence $\{(s_n,s_n^*)\}_{n \ge 1}$ of elements of $G(S)$ such that\break $\lim_n\dist(s_n,W) = 0$ and $\lim_n\|s_n^* - w^*\| = 0$.   From \Lem{DISTlem}(b),  there exist $w \in W$ and a subnet $\{(s_\beta,s_\beta^*)\}$ of $\{(s_n,s_n^*)\}$ such that $s_\beta \to w$ in $w(E,E^*)$.   Obviously, $s_\beta^* \to w^*$ in ${\cal T}_{E^*}$ and so, from \Lem{NORMWKlem}(b), $(w,w^*) \in G(S)$.   So $w^* \in S(W)$, as required.
\end{proof}
\section{Sequential characterizatons of the Fitzpatrick extension}\label{FITZCHARsec}
The main result of this section, \Thm{SEQCHARthm}, contains three characterizations of the Fitzpatrick extension of a closed, monotone multifunction, two of them in terms of a {\em sequence} of elements from its graph.   \Thm{SEQCHARthm} is bootstrapped from \Lem{EXSEQlem}, which depends on the three preceding lemmas, about which we now make a few comments.
\par
The genesis for \Lems{ALLFlem} and \ref{SQlem} below is ultimately the sharpening by Gossez of a result established by Rockafellar in one of his proofs of the maximal monotonicity of subdifferentials  \big(see \cite[Lemma~3.1, pp.\ 376--377]{GOSSEZ} and \cite[Proposition 1, pp.\ 211--212]{RTRMMT}\big).   \Lem{FITZCHARlem} is a special case of a result from \cite{PARTONE}.   In what follows, we write  $\bigvee_{i = 0}^mf_i$ for $\max\{f_0, \dots, f_m\}$.  
\begin{lemma}\label{ALLFlem}
Let $X$ be a nonzero Banach space, $m \ge 1$, $f_0 \in \PCLSC(X)$ and $f_1,\dots,f_m$ be real, convex, continuous functions on $X$.   Suppose that there exists $x\dbs \in E\dbs$ such that, for all $i = 0, \dots, m$, ${f_i}\dbs(x\dbs) \le 0$.   Then, for all $\eps > 0$,    there exists $x \in X$ such that, for all $i = 0, \dots, m$, $f_i(x) \le \eps$.
\end{lemma}
\begin{proof}
We first observe that
\begin{equation*}
\ts\inf_X \bigvee_{i = 0}^m{f_i} = \bra{0}{x\dbs} -  \big(\bigvee_{i = 0}^m{f_i}\big)^*(0) \le \big(\bigvee_{i = 0}^m{f_i}\big)\dbs(x\dbs).
\end{equation*}
From \cite[Corollary 45.5, p.\ 174]{HBM} or \cite[Corollary~7, p.\ 3558]{FS}, $\ts\big(\bigvee_{i = 0}^mf_i\big)\dbs(x\dbs) = \bigvee_{i = 0}^m{f_i}\dbs(x\dbs)$ and, by hypothesis, $\bigvee_{i = 0}^m{f_i}\dbs(x\dbs) \le 0$.   Thus $\inf_X \bigvee_{i = 0}^m{f_i} \le 0$.   This gives the desired result.  
\end{proof}
\begin{lemma}\label{SQlem}
Let $f_0 \in \PCLSC(E \times E^*)$, $z\dbs \in E\dbs$, $f_0\dbs\big(z\dbs,0\big) \le 0$ and $z^* \in E^*$.   Then, for all $n \in \NN$, there exists $(x_n,x_n^*) \in E \times E^*$ such that   %
\begin{gather*}
f_0(x_n,x_n^*) \le 1/n^2,\\
\|x_n\| \le \|z\dbs\| + 1/n^2,\ \|x_n^*\| \le 1/n^2,\hbox{ and }|\bra{x_n}{z^*} - \bra{z^*}{z\dbs}| \le 1/n^2.
\end{gather*}

\end{lemma}
\begin{proof}
Define the real, continuous convex functions $f_1,f_2,f_3,f_4$ on $E \times E^*$ by $f_1(x,x^*) := \|x\| - \|z\dbs\|$, $f_2(x,x^*) := \|x^*\|$, $f_3(x,x^*) := \bra{x}{z^*} - \bra{z^*}{z\dbs}$ and $f_4(x,x^*) := \bra{z^*}{z\dbs} - \bra{x}{z^*}$.   By direct computation, for all $i = 0, \dots, 4$, ${f_i}\dbs\big(z\dbs,0\big) \le 0$ (with equality when $i = 1, \dots, 4$).   The result now follows from \Lem{ALLFlem} with $X = E \times E^*$ and $\eps = 1/n^2$.
\end{proof}
\begin{lemma}\label{FITZCHARlem}
Let $g \in \PCLSC(E^* \times E\dbs)$,
%
\begin{equation}\label{FITZCHAR2}
(y^*,y\dbs) \in E^* \times E\dbs \qlr g(y^*,y\dbs) \ge \bra{y^*}{y\dbs},
\end{equation}
$(z^*,z\dbs) \in E \times E\dbs$, and
\begin{equation}\label{FITZCHAR1}
g(z^*,z\dbs) = \bra{z^*}{z\dbs}.
\end{equation}
Then
\begin{equation}\label{FITZCHAR3}
g^*(z\dbs,\wh{z^*}) = \bra{z^*}{z\dbs}.
\end{equation}
\end{lemma}
\begin{proof}
Let $(y^*,y\dbs) \in E^* \times E\dbs$ and $\lambda \in \,]0,1[$.   From \eqref{FITZCHAR1}, the convexity of $g$ and \eqref{FITZCHAR2},    
\begin{align*}
\lambda g(y^*,y\dbs)
&= \lambda g(y^*,y\dbs) + (1 - \lambda)g(z^*,z\dbs) - (1 - \lambda)\bra{z^*}{z\dbs}\\
&\ge g\big(\lambda y^* + (1 - \lambda)z^*,\lambda y\dbs + (1 - \lambda)z\dbs\big) - (1 - \lambda)\bra{z^*}{z\dbs}\\
&\ge \Bra{\lambda y^* + (1 - \lambda)z^*}{\lambda y\dbs + (1 - \lambda)z\dbs} - (1 - \lambda)\bra{z^*}{z\dbs}\\
&= \lambda^2\bra{y^*}{y\dbs} + \lambda(1 - \lambda)\big[\bra{y^*}{z\dbs} + \bra{z^*}{y\dbs} - \bra{z^*}{z\dbs}\big]. 
\end{align*}
Thus, dividing by $\lambda$,
\begin{equation*}
g(y^*,y\dbs) \ge \lambda\bra{y^*}{y\dbs} + (1 - \lambda)\big[\bra{y^*}{z\dbs} + \bra{z^*}{y\dbs} - \bra{z^*}{z\dbs}\big].
\end{equation*}
Letting $\lambda \to 0$ and rearranging the terms,
\begin{equation*}
\bra{y^*}{z\dbs} + \bra{z^*}{y\dbs} - g(y^*,y\dbs) \le \bra{z^*}{z\dbs},
\end{equation*}
that is to say,
\begin{align*}
\Bra{(y^*,y\dbs)}{(z\dbs,\wh{z^*})} - g(y^*,y\dbs) \le \bra{z^*}{z\dbs}. 
\end{align*}
Taking the supremum over $(y^*,y\dbs) \in E^* \times E\dbs$,\quad  $g^*(z\dbs,\wh{z^*}) \le \bra{z^*}{z\dbs}$.\quad   On the other hand, from \eqref{FITZCHAR1} again,
\begin{align*}
g^*(z\dbs,\wh{z^*})
&\ge \Bra{(z^*,z\dbs)}{(z\dbs,\wh{z^*})} - g(z^*,z\dbs)\\
&= 2\bra{z^*}{z\dbs} - \bra{z^*}{z\dbs} = \bra{z^*}{z\dbs}. 
\end{align*}
This gives \eqref{FITZCHAR3}, and completes the proof of \Lem{FITZCHARlem}.   (This proof is based partly on the proof of \cite[Lemma 19.12, p.\ 82]{HBM}.)
\end{proof}
\begin{lemma}\label{EXSEQlem}
Let $T\colon\ E \toto E^*$ be closed monotone and quasidense and\break $(z^*,z\dbs) \in E^* \times E\dbs$.   Then {\em (a)$\ifff$(b)$\lr$(c)$\lr$(d)}:
\par
\noindent
{\rm(a)}\enspace $(z^*,z\dbs) \in G(T^\F)$.\par
\noindent
{\rm(b)}\enspace ${\varphi_T}^*(z^*,z\dbs) = \bra{z^*}{z\dbs}$.
\par
\noindent
{\rm(c)}\enspace ${\varphi_T}\dbs(z\dbs,\wh{z^*}) = \bra{z^*}{z\dbs}$.\par
\noindent
{\rm(d)}\enspace There exists a sequence $\{(t_n,t_n^*)\}_{n \ge 1}$ of elements of $G(T)$ such that\break $\lim_n\bra{t_n}{t_n^*} = \bra{z^*}{z\dbs}$ and $\lim_n\|t_n^* - z^*\| = 0$.
\end{lemma}
\begin{proof}
It is immediate from \eqref{FITZ1} that (a)$\ifff$(b).   It is also immediate from \Lem{FITZCHARlem} with $g := {\varphi_T}^*$ and \Thm{PHTHthm} that (b)$\lr$(c).   Now suppose that (c) is true.  Let $S := T - z^*$.    Clearly, $S$ is closed monotone and quasidense.   It is also easily seen that ${\varphi_S}\dbs\big(z\dbs,0\big) = 0$.\quad From \Lem{SQlem} with $f_0 = \varphi_S$, for all $n \in \NN$, there exists $(x_n,x_n^*) \in E \times E^*$ such that
\begin{gather}
\varphi_S(x_n,x_n^*) \le 1/n^2,\label{EXSEQ1}\\
\|x_n\| \le \|z\dbs\| + 1/n^2,\ \|x_n^*\| \le 1/n^2,\hbox{ and }|\bra{x_n}{z^*} - \bra{z^*}{z\dbs}| \le 1/n^2.
\label{EXSEQ2}
\end{gather}
From \Def{QDdef}, there exists $(s_n,s_n^*) \in G(S)$ such that
\begin{equation}\label{EXSEQ7}
\half\|s_n - x_n\|^2 + \half\|s_n^* - x_n^*\|^2 + \bra{s_n - x_n}{s_n^* - x_n^*} \le 1/n^2.
\end{equation}
Set $M := \|z\dbs\| + 1$.   From \eqref{EXSEQ2}, $\|x_n\| \le M$.  From \eqref{EXSEQ1} and \Def{PHdef}, $\bra{s_n}{x_n^*} + \bra{x_n}{s_n^*} - \bra{s_n}{s_n^*} \le 1/n^2$.   Combining this with \eqref{EXSEQ2}, we see that $\bra{s_n - x_n}{s_n^* - x_n^*} \ge \bra{x_n}{x_n^*} - 1/n^2 \ge -M/n^2 - 1/n^2$, and so \eqref{EXSEQ7} gives\quad $\half\|s_n - x_n\|^2 + \half\|s_n^* - x_n^*\|^2 \le M/n^2 + 2/n^2$.\quad Thus
\begin{equation}\label{EXSEQ8}
\limn_n\|s_n - x_n\| = 0\quand\limn_n\|s_n^* - x_n^*\| = 0.
\end{equation}
Combining \eqref{EXSEQ8} with \eqref{EXSEQ2}, $\sup_n\|s_n\| < \infty$ and $\lim_n\|s_n^*\| = 0$, from which $\lim_n\bra{s_n}{s_n^*} = 0$.   Now $\bra{s_n}{z^*} = \bra{s_n - x_n}{z^*} + \bra{x_n}{z^*}$ and so, from \eqref{EXSEQ8} and \eqref{EXSEQ2} again, $\lim_n\bra{s_n}{z^*} = \bra{z^*}{z\dbs}$.  Let $(t_n,t_n^*) := (s_n,s_n^* + z^*) \in G(T)$.   Clearly, $\lim_n\|t_n^* - z^*\| = \lim_n\|s_n^*\| = 0$ and 
$\lim_n\bra{t_n}{t_n^*} = \lim_n\bra{s_n}{s_n^* + z^*} = \lim_n\bra{s_n}{s_n^*} + \lim_n\bra{s_n}{z^*} = \bra{z^*}{z\dbs}$.
This completes the proof of (d).
\end{proof}
\begin{theorem}[Sequential characterizations of the Fitzpatrick extension]\label{SEQCHARthm}
Let $S\colon\ E \toto E^*$ be closed monotone and quasidense and $(z^*,z\dbs) \in E^* \times E\dbs$.   Then the following four conditions are equivalent:\par
\noindent
{\rm(a)}\enspace $(z^*,z\dbs) \in G(S^\F)$.\par
\noindent
{\rm(b)}\enspace ${\varphi_S}\dbs\big(z\dbs,\wh{z^*}\big) = \bra{z^*}{z\dbs}$.\par
\noindent
{\rm(c)}\enspace For all $w^* \in E^*$, there exists a sequence $\{(s_n,s_n^*)\}_{n \ge 1}$ of elements of $G(S)$ such that $\bra{s_n}{s_n^* - w^*} \to \bra{z^* - w^*}{z\dbs}$ and $\|s_n^* - z^*\| \to 0$ as $n \to \infty$.
\par
\noindent
{\rm(d)}\enspace For all $(w,w^*) \in E \times E^*$, there exists a sequence $\{(s_n,s_n^*)\}_{n \ge 1}$ of elements of $G(S)$ such that
\begin{equation}\label{SEQCHAR3}
\bra{s_n - w}{s_n^* - w^*} \to \bra{z^* - w^*}{z\dbs - \wh w}\hbox{ and }\|s_n^* - z^*\| \to 0\hbox{ as }n \to \infty.
\end{equation}
\end{theorem}
\begin{proof}
It is immediate from the argument already used in \Lem{EXSEQlem} that (a)$\lr$(b).
\par
Now suppose that (b) is true and $(w,w^*) \in E \times E^*$.   Let $T := S - w^*$.   Clearly, $T$ is closed monotone and quasidense.\par   By direct computation,\quad ${\varphi_S}\dbs\big(z\dbs,\wh{z^*}\big) = {\varphi_T}\dbs\big(z\dbs,\wh{z^*} - \wh{w^*}\big) + \bra{w^*}{z\dbs}$.\quad   From (b), ${\varphi_T}\dbs\big(z\dbs,\wh{z^*} - \wh{w^*}\big) = \bra{z^* - w^*}{z\dbs}$.   From \Lem{EXSEQlem}\big((c)$\lr$(d)\big), there exists a sequence $\{(t_n,t_n^*)\}_{n \ge 1}$ of elements of $G(T)$ such that\break $\lim_n\bra{t_n}{t_n^*} = \bra{z^* - w^*}{z\dbs}$ and $\lim_n\|t_n^* - (z^* - w^*)\| = 0$.   Let\break $(s_n,s_n^*) := (t_n,t_n^* + w^*) \in G(S)$.   Then $\lim_n\bra{s_n}{s_n^* - w^*} = \bra{z^* - w^*}{z\dbs}$ and\break $\lim_n\|s_n^* - z^*\| = 0$, giving (c).
\par
Now suppose that (c) is true.   Then, for all $(w,w^*) \in E \times E^*$,
\begin{align*}
\bra{s_n - w}{s_n^* - w^*} &- \bra{z^* - w^*}{z\dbs - \wh w}\\
&= \bra{s_n}{s_n^* - w^*} - \bra{w}{s_n^* - z^*} + \bra{w^* - z^*}{z\dbs}\\
&\to \bra{z^* - w^*}{z\dbs} + 0  + \bra{w^* - z^*}{z\dbs} = 0.
\end{align*}
Thus \eqref{SEQCHAR3} is satisfied, and so (d) is true.
\par
Suppose, finally, that (d) is true.    Then, reversing the argument above,
\begin{equation*}
\bra{s_n}{s_n^* - w^*} - \bra{w}{s_n^* - z^*} + \bra{w^* - z^*}{z\dbs} \to 0,
\end{equation*}
from which
\begin{align*}
\bra{w}{z^*} + \bra{w^*}{z\dbs} &- \bra{z^*}{z\dbs}
= \limn_n\big[\bra{s_n}{w^*} + \bra{w}{s_n^*} - \bra{s_n}{s_n^*}\big]\\
&\le \supn_{(s,s^*) \in G(S)}\big[\bra{s}{w^*} + \bra{w}{s^*} - \bra{s}{s^*}\big] = \varphi_S(w,w^*).
\end{align*}
Consequently,\quad $\bra{w}{z^*} + \bra{w^*}{z\dbs} - \varphi_S(w,w^*) \le \bra{z^*}{z\dbs}$.   Taking the\break supremum over $(w,w^*)$,\quad ${\varphi_S}^*(z^*,z\dbs) \le \bra{z^*}{z\dbs}$,\quad and it follows from\break \Thm{PHTHthm} and \eqref{FITZ1} that (a) is true.
\end{proof}
\begin{remark}\label{STRONGERrem}
The equivalence of (a) and (b) above was established in\break\cite[Lemma 12.4(a), p.\ 1047]{PARTONE}.
\end{remark}
\section{Type (FP)}\label{FPsec}
\Lem{FPlem} below will simplify the computations in \Thm{FPthm} considerably:
\begin{lemma}\label{FPlem}
Let $T\colon\ E \toto E^*$ be closed, monotone and quasidense, $\VT$ be an open convex subset of $E^*$, $\VT \ni 0$, $\VT \cap R(T) \neq \emptyset$ and
\begin{equation}\label{FP1}
(t,t^*) \in G(T)\ \hbox{and}\ t^* \in \VT \qlr \bra{t}{t^*} \ge 0.
\end{equation}
Then $(0,0) \in G(T)$.  
\end{lemma}
\begin{proof}
Let $y^* \in \VT \cap R(T)$.   Since the segment $[0,y^*]$ is a compact subset of the open set $\VT$, we can choose $\eps > 0$ so that\quad$\KT := [0,y^*] + \eps E^*_1 \subset \VT$.\quad
From \eqref{DNORMAL3},  $R(T) \cap \intr\,R(\partial\tauKT) = R(T) \cap \intr\,\KT \ni y^*$.  We now define the\break multifunction $P\colon\ E \toto E^*$ by $P(y) := (T^\F + \partial\tauKT^\F)^{-1}(\wh y)$.  \eqref{DNORMAL2} and \Thm{STRthm}\break imply that $P$ is closed and quasidense.   Let $\eta > 0$.   Then there exists\break $(y,z^*) \in G(P)$ such that
\begin{equation}\label{FP2}
\half\|y\|^2 + \half\|z^*\|^2 + \bra{y}{z^*} < \eta.
\end{equation}
We can choose $z\dbs \in T^\F(z^*)$ such that $\wh y - z\dbs \in {\partial\tauKT}^\F(z^*)$. From \Lem{FITZPARTAUlem}, $z^* \in \KT \subset \VT$ and $\bra{z^*}{\wh y - z\dbs} = \sup\Bra{\KT}{\wh y - z\dbs} \ge \eps\|\wh y - z\dbs\| \ge 0$.   Thus
\begin{equation}\label{FP3}
\bra{y}{z^*} = \bra{z^*}{\wh y} \ge \eps\|\wh y - z\dbs\| + \bra{z^*}{z\dbs} \ge  \bra{z^*}{z\dbs}.
\end{equation}
\par
From \Lem{EXSEQlem}, there exists a sequence $\{(t_n,t_n^*)\}_{n \ge 1}$ of elements of $G(T)$ such that\quad$\bra{t_n}{t_n^*} \to \bra{z^*}{z\dbs}$\quad and\quad $\|t_n^* - z^*\| \to 0$ as $n \to \infty$.\quad If $n$ is\break sufficiently large, $t_n^* \in \VT$, and so, from \eqref{FP1}, $\bra{t_n}{t_n^*} \ge 0$.   Passing to the limit, $\bra{z^*}{z\dbs} \ge 0$.  Combining this with \eqref{FP2} and \eqref{FP3},
\begin{equation*}
\half\|y\|^2 + \half\|z^*\|^2 < \eta.
\end{equation*}
Taking $\eta$ arbitrarily small and using the fact that $P$ is closed, we derive that $(0,0) \in G(P)$.   Repeating the argument already used above, we can choose $z_0\dbs \in T^\F(0)$ such that $\bra{0}{0} \ge \eps\|0 - z_0\dbs\|$.  Thus $z_0\dbs = 0$, from which\break $(0,0) = (0,z_0\dbs) \in G(T^\F)$, and so \eqref{FITZ2} implies that $(0,0) \in G(T)$.
\end{proof}
\begin{definition}\label{FPdef}
Let $S\colon\ E \toto E^*$ be monotone.   We say that $S$ is {\em of type (FP)} or {\em locally maximally monotone} if whenever $\UT$ is a convex open subset of $E^*$, $\UT \cap R(S) \neq \emptyset$, $(w,w^*) \in E \times \UT$ and
\begin{equation}\label{FP41}
(s,s^*) \in G(S)\ \hbox{and}\ s^* \in \UT \qlr \bra{s - w}{s^* - w^*} \ge 0,
\end{equation}
then $(w,w^*) \in G(S)$.  \big(If we take $\UT = E^*$, we see that every monotone multifunction of type (FP) is maximally monotone.\big)
\end{definition}
\begin{theorem}[The type (FP) criterion for quasidensity]\label{FPthm}
Let $S\colon\ E \toto E^*$ be maximally monotone.  Then the conditions {\em(a)}, {\em(b)} and {\em(c)} are equivalent.
\par\noindent
{\rm(a)}\enspace $S$ is quasidense.
\par\noindent
{\rm(b)}\enspace $S$ is of type (FP).  
\par\noindent
{\rm(c)}\enspace For all $(w^*,w\dbs) \in E^* \times E\dbs$, $\infn_{(s,s^*) \in G(S)}\bra{s^* - w^*}{\wh s - w\dbs} \le 0$.
\end{theorem}
\begin{proof}
(a)$\lr$(b).   Let $\UT$ be an open convex subset of $E^*$, $\UT \cap R(S) \neq \emptyset$, $(w,w^*) \in E \times \UT$ and \eqref{FP41} be satisfied.   It follows easily from \Lem{FPlem} with $T$ defined so that $G(T) = G(S) - (w,w^*)$ and $\VT := \UT - w^*$ that $(w,w^*) \in G(S)$.
\smallbreak
(b)$\lr$(c).   Let $S$ be of type (FP) and $\infn_{(s,s^*) \in G(S)}\bra{s^* - w^*}{\wh s - w\dbs} > 0$.\break   We choose $\eps > 0$ so that $\infn_{(s,s^*) \in G(S)}\bra{s^* - w^*}{\wh s - w\dbs} > \eps$ and define\break $\eta := \eps/(2\|w\dbs\| + 2)$. 
Let $y^* \in R(T)$.   From \Lem{ALLFlem} with $X = E$, $f_0 := \|\cdot\| - \|w\dbs\|$ and $f_1 := y^* - \bra{y^*}{w\dbs}$ \big(so that ${f_0}\dbs(w\dbs) = {f_1}\dbs(w\dbs) = 0$\big) there exists $w \in E$ such that
\begin{equation}\label{THIN1}
\|w\| \le \|w\dbs\| + 1\hbox{ and }\bra{y^*}{\wh w - w\dbs} \le \eta.
\end{equation}
Let $T\colon\ E \toto E^*$ be defined so that $G(T) = G(S) - (w,w^*)$.   Then we have
%
\begin{equation}\label{THIN3}
(t,t^*) \in G(T) \qlr \bra{t^*}{\wh t + \wh w - w\dbs} > \eps.
\end{equation}
Let $\UT := [0,y^*] + \big\{z^* \in E^*\colon\ \|z^*\| < \eta\}$. $\UT$ is convex and open and  $\UT \cap R(T) \neq \emptyset$.   We now prove that
\begin{equation}\label{THIN2}
(t,t^*) \in G(T)\ \hbox{and}\ t^* \in \UT \qlr \bra{t}{t^*} \ge 0.
\end{equation}
To this end, let $(t,t^*) \in G(T)$ and $t^* \in \UT$.    Then there exists $\lambda \in [0,1]$ such that $\|t^* - \lambda y^*\| < \eta$.   Combining this with \eqref{THIN1} and \eqref{THIN3},
\begin{align*}
\bra{t}{t^*}
&= \bra{t^*}{\wh t + \wh w - w\dbs} - \bra{t^*}{\wh w - w\dbs} > \eps - \bra{t^*}{\wh w - w\dbs}\\
&= \eps - \bra{t^* - \lambda y^*}{\wh w - w\dbs} - \lambda\bra{y^*}{\wh w - w\dbs} \ge \eps - \|\wh w - w\dbs\|\eta - \lambda\eta\\
&\ge\eps - (2\|w\dbs\| + 1)\eta - \lambda\eta \ge \eps - (2\|w\dbs\| + 2)\eta = 0.
\end{align*}
This completes the proof of \eqref{THIN2}.
Clearly, $T$ is of type (FP) and so, from\break \Def{FPdef}, $(0,0) \in G(T)$.   But then \eqref{THIN3} would give $\bra{0}{\wh w - w\dbs} > \eps$, which is impossible.  
\smallbreak
(c)$\lr$(a).   From \Lem{PHISlem}, for all $(w^*,w\dbs) \in E^* \times E\dbs$, 
\begin{align*}
\infn_{(s,s^*) \in G(S)}&\bra{s^* - w^*}{\wh s - w\dbs}\\
&= \bra{w^*}{w\dbs} + \infn_{(s,s^*) \in G(S)}[\bra{s}{s^*} - \bra{s}{w^*} - \bra{s^*}{w\dbs}]\\
&= \bra{w^*}{w\dbs} + \infn_{(s,s^*) \in G(S)}[\varphi_S(s,s^*) - \bra{s}{w^*} - \bra{s^*}{w\dbs}]\\
&\ge \bra{w^*}{w\dbs} + \infn_{(x,x^*) \in E \times E^*}[\varphi_S(x,x^*) - \bra{x}{w^*} - \bra{x^*}{w\dbs}]\\
&= \bra{w^*}{w\dbs} - {\varphi_S}^*(w^*,w\dbs). 
\end{align*}
Thus (c) implies that, for all $(w^*,w\dbs) \in E^* \times E\dbs$, ${\varphi_S}^*(w^*,w\dbs) \ge \bra{w^*}{w\dbs}$, and it follows from \Thm{PHTHthm} that $S$ is quasidense.
\end{proof}
\begin{remark}\label{NIrem}
Condition (c) above is exactly that {\em $S$ is of type (NI).}   So the fact that (a)$\ifff$(c) above can easily be deduced from the results proved by Marques Alves and Svaiter in \cite[Theorem~1.2(1$\ifff$5), p.\ 885]{ASS} or Voisei and Z\u{a}linescu in\break \cite[Theorem 4.1, pp.\ 1027--1028]{VZ}.   Of course the conditions contained in\break \cite[Theorem~1.2 (3 and 4)]{ASS} are closely related to our definition of quasidensity.   See \cite[Theorem 6.10, p.\ 1031]{PARTONE} for a more general result.  The implication (b)$\lr$(c) above was established by Bauschke, Borwein, Wang and Yao in\break \cite[Theorem~3.1, pp.\ 1878--1879]{BBWYFP}.   It would be nice to find a proof of (a)$\lr$(b) above free of the complexities of \Sec{FITZCHARsec}, but this seems a hard problem.
\par
So the equivalences outlined in \Thm{FPthm} are already in the literature, but the approach outlined in this paper shows that \Thms{FUZZDthm} and \ref{FUZZEthm} give additional information about these classes of maximal monotone multifunctions and the equivalent classes ``type (D)'', ``dense type'' and ``type (ED)'', as well as those that satisfy the ``negative alignment criterion'' of \cite[Theorem 11.6, p.\ 1045]{PARTONE}.
\section{Appendix 1}\label{APPsec}
In Appendix 1, we discuss the function $\theta_S$ briefly, show the connection with the Gossez extension, and give a self--contained proof of \Thm{AFMAXthm}.
\par 
Let $S$ be closed, monotone and quasidense. We define the function\break $\theta_S\colon\ E^* \times E\dbs \to \rbar$ by
\begin{equation*}
\theta_S(w^*,w\dbs):= \supn_{(s,s^*) \in G(S)}\big[\bra{s}{w^*} + \bra{s^*}{w\dbs} - \bra{s}{s^*}\big].
\end{equation*}
Then condition \Thm{FPthm}(c) can be put in the equivalent form:
\begin{equation}\label{NI1}
\all\ (w^*,w\dbs) \in E^* \times E\dbs,\ \theta_S(w^*,w\dbs) \ge \bra{w^*}{w\dbs}.
\end{equation} 
Now let $(w^*,w\dbs) \in E^* \times E\dbs$.  An examination of the proof that (c)$\lr$(a) in \Thm{FPthm} shows that\quad ${\varphi_S}^*(w^*,w\dbs) \ge \theta_S(w^*,w\dbs)$.\quad  See \cite[Eq. (21),\break p.\ 1029]{PARTONE}.   We also have
\begin{align*}
{\theta_S}^*(w\dbs,\wh{w^*}) &= \supn_{(x^*,x\dbs) \in E^* \times E\dbs}\big[\bra{x^*}{w\dbs} + \bra{w^*}{x\dbs} - \theta_S(x^*,x\dbs)\big]\\
&\ge \supn_{(x,x^*) \in E \times E^*}\big[\bra{x^*}{w\dbs} + \bra{x}{w^*} - \theta_S(x^*,\wh x)\big].
\end{align*}
Since $\theta_S(x^*,\wh x) = \varphi_S(x,x^*)$,\quad it follows that\quad  ${\theta_S}^*(w\dbs,\wh{w^*}) \ge {\varphi_S}^*(w^*,w\dbs)$.\quad Consequently, using \eqref{NI1},
\begin{equation}\label{NI2}
{\theta_S}^*(w\dbs,\wh{w^*}) \ge {\varphi_S}^*(w^*,w\dbs) \ge \theta_S(w^*,w\dbs) \ge \bra{w^*}{w\dbs}.
\end{equation} 
From \Lem{FITZCHARlem}, $\theta_S(z^*,z\dbs) = \bra{z^*}{z\dbs}$ implies that ${\theta_S}^*\big(z\dbs, \wh{z^*}\big) = \bra{z^*}{z\dbs}$.   Combining this with \eqref{NI2} and \eqref{FITZ1}, we see that
\begin{align}
G(S^\F) &:= \{(w^*,w\dbs)\colon\ {\varphi_S}^*(w^*,w\dbs) = \bra{w^*}{w\dbs}\}\notag\\
&= \{(w^*,w\dbs)\colon\ \theta_S(w^*,w\dbs) = \bra{w^*}{w\dbs}\}\label{NI3}\\
&= \{(w^*,w\dbs)\colon\ {\theta_S}^*\big(w\dbs,\wh{w^*}\big) = \bra{w^*}{w\dbs}\}.\label{NI4}
\end{align}
In particular, using the definition of $\theta_S$, $(y^*,y\dbs) \in G(S^\F)$ exactly when $(y\dbs,y^*)$ is in the {\em Gossez extension} of $G(S)$ \big(see \cite[Lemma~2.1,p. 275]{GOSSEZ}\big).
\par
Finally, we show how \eqref{NI3} leads to a proof of \Thm{AFMAXthm}.   To this end, let $(w^*,w\dbs) \in E^* \times E\dbs$ and $\inf_{(z^*,z\dbs) \in G(S^\F)}\bra{z^* - w^*}{z\dbs - w\dbs} \ge 0$.   From \eqref{FITZ2}, $\inf_{(t,t^*) \in G(S)}\bra{t^* - w^*}{\wh t - w\dbs} \ge 0$. It follows from \Thm{FPthm}\big((a)$\lr$(c)\big) that $\inf_{(t,t^*) \in G(S)}\bra{t^* - w^*}{\wh t - w\dbs} = 0$, and \eqref{NI3} now implies that $(w^*,w\dbs) \in G(S^\F)$.
\end{remark}
\section{Appendix 2}\label{FPALTsec}
In Appendix 2, we give a proof of \Lem{FPlem} that does not use \Thm{STRthm}, but uses instead Rockafellar's formula for the conjugate of a sum and version of the Fenchel duality theorem.
\begin{lemma}\label{Tlem}
Let $T\colon\ E \toto E^*$ be closed, monotone and quasidense and $\KT$ be a $w(E^*,E)$--compact convex subset of $E^*$ such that $R(T) \cap \intr \KT \ne \emptyset$.   Then there exist $z^* \in \KT$ and $z\dbs,x\dbs \in E\dbs$ such that
\begin{gather}
{\varphi_T}^*(z^*,z\dbs) + \sup\bra{\KT}{x\dbs - z\dbs} + \half\|(z^*,x\dbs)\|^2 \le 0,\label{T1}\\
\noalign{\noindent and}
z\dbs \in T^\F(z^*)\hbox{ and }\bra{z^*}{z\dbs} + \sup\bra{\KT}{x\dbs - z\dbs} + \half\|(z^*,x\dbs)\|^2 \le 0.\label{T2}
\end{gather}  
\end{lemma}
\begin{proof}
For all $(x,x^*) \in E \times E^*$, let $h(x,x^*) := \IKT(x^*)$.   Clearly, $\dom\,h = E \times \KT$, and so $h \in \PCLSC(E \times E^*)$.   Also, for all $(y^*,y\dbs) \in E^* \times E\dbs$,
\begin{equation}\label{T3}
h^*(y^*,y\dbs) = \I_{\{0\}}(y^*) + \sup\bra{\KT}{y\dbs}.
\end{equation}  
If $(t,t^*) \in G(T)$ and $t^* \in \intr\,\KT$ then, from \Lem{PHISlem}, $\varphi_T(t,t^*) = \bra{t}{t^*} \in \RR$ and so $\dom\,\varphi_T \cap \intr\,\dom\,h = \dom\,\varphi_T \cap (E \times \intr\,\KT) \ne \emptyset$.   Thus, from Rockafellar's formula for the conjugate of a sum, \cite[Theorem~3(a), pp.\ 85--86]{FENCHEL}, and \eqref{T3}, for all $(z^*,x\dbs) \in E^* \times E\dbs$,
\begin{align}
&\supn_{(y,x^*) \in E \times \KT}\big[\bra{y}{z^*} + \bra{x^*}{x\dbs} - \varphi_T(y,x^*)\big]\notag\\
&\enspace= \supn_{(y,x^*) \in E \times E^*}\big[\bra{y}{z^*} + \bra{x^*}{x\dbs} - (\varphi_T + h)(y,x^*)\big]\notag\\
&\enspace= (\varphi_T + h)^*(z^*,x\dbs).\notag\\
&\enspace= \minn_{(w^*,z\dbs) \in E^* \times E\dbs}\big[{\varphi_T}^*(w^*,z\dbs) + h^*(z^* - w^*,x\dbs - z\dbs)\big]\notag\\
&\enspace= \minn_{(w^*,z\dbs) \in E^* \times E\dbs}\big[{\varphi_T}^*(w^*,z\dbs) + \I_{\{0\}}(z^* - w^*) + \sup\bra{\KT}{x\dbs - z\dbs}\big]\notag\\
&\enspace= \minn_{z\dbs \in E\dbs}\big[{\varphi_T}^*(z^*,z\dbs) + \sup\bra{\KT}{x\dbs - z\dbs}\big].\label{T4}
\end{align}
For all $(x,x^*) \in E \times E^*$, let
\begin{equation*}
f(x,x^*) :=
\begin{cases}
\inf_{y \in E}\big[\varphi_T(y,x^*) + \max\bra{x - y}{\KT}\big]
&(x^* \in \KT);\\
\infty
&(x^* \not\in \KT).
\end{cases}
\end{equation*}
Since $\KT$ is a $w(E^*,E)$--closed convex subset of $E^*$, we have
\begin{equation}\label{T5}
\supn_{w \in E}\big[\bra{w}{z^*} - \max\bra{w}{\KT}\big] = \IKT(z^*)
\end{equation} 
If $x^* \in \KT$ then, from \Lem{PHISlem}, for all $y \in E$, $\varphi_T(y,x^*) + \max\bra{x - y}{\KT} \ge \bra{y}{x^*} + \bra{x - y}{x^*} = \bra{x}{x^*}$ thus,
\begin{equation}\label{T6}
\all\ (x,x^*) \in E \times E^*,\quad f(x,x^*) \ge \bra{x}{x^*}.
\end{equation}
Thus $f\colon\ E \times E^* \to \rbar$ and $f$ is easily seen to be convex.   On the other hand, if $(t,t^*) \in G(T)$ and $t^* \in \KT$ then (taking $y = t$),
\begin{equation*}
f(t,t^*) \le \varphi_T(t,t^*) + \max\bra{t - t}{\KT} =  \bra{t}{t^*} + 0 = \bra{t}{t^*}.
\end{equation*}
Thus $f$ is proper.  From \eqref{T4} and \eqref{T5},
\begin{align}
f^*&(z^*,x\dbs)\notag\\
&= \sup_{(x,x^*) \in E \times \KT,\ y \in E}\big[\bra{x}{z^*} + \bra{x^*}{x\dbs} - \varphi_T(y,x^*) - \max\bra{x - y}{\KT}\big]\notag\\
&= \sup_{(w,x^*) \in E \times \KT,\ y \in E}\big[\bra{w + y}{z^*} + \bra{x^*}{x\dbs} - \varphi_T(y,x^*) - \max\bra{w}{\KT}\big]\notag\\
&= \sup_{(y,x^*) \in E \times \KT,\ w \in E}\big[\bra{y}{z^*} + \bra{x^*}{x\dbs} - \varphi_T(y,x^*) + \bra{w}{z^*} - \max\bra{w}{\KT}\big]\notag\\
&= \min_{z\dbs \in E\dbs}\big[{\varphi_T}^*(z^*,z\dbs) + \sup\bra{\KT}{x\dbs - z\dbs}\big] + \IKT(z^*).\label{T8}
\end{align}
From \eqref{T6}, 
\begin{equation*}
\all\ (x,x^*) \in E \times E^*,\quad f(x,x^*) + \half\|(x,x^*)\|^2 \ge \bra{x}{x^*} + \half\|(x,x^*)\|^2 \ge 0.
\end{equation*}
Rockafellar's version of the Fenchel duality theorem, \cite[Theorem~3(a), p.\ 85]{FENCHEL}, gives $(z^*,x\dbs) \in E^* \times E\dbs$ such that $f^*(z^*,x\dbs) + \half\|-(z^*,x\dbs)\|^2 \le 0$.   From \eqref{T8}, $z^* \in \KT$ and there exists $z\dbs \in E\dbs$ such  that \eqref{T1} is satisfied.   It follows from this that\quad ${\varphi_T}^*(z^*,z\dbs) + \bra{z^*}{x\dbs - z\dbs} + \half\|(z^*,x\dbs)\|^2 \le 0$.\quad Since\quad $\bra{z^*}{x\dbs} + \half\|(z^*,x\dbs)\|^2 \ge 0$,\quad this implies that\quad  ${\varphi_T}^*(z^*,z\dbs) - \bra{z^*}{z\dbs} \le 0$.\quad From \Thm{PHTHthm}, and \eqref{FITZ1}, ${\varphi_T}^*(z^*,z\dbs) = \bra{z^*}{z\dbs}$\quad and \quad $z\dbs \in T^\F(z^*)$\quad and so \eqref{T2} follows from \eqref{T1}.
\end{proof}    
Here is the promised proof of \Lem{FPlem}.
\begin{lemma}\label{FPALTlem}
Let $T\colon\ E \toto E^*$ be closed, monotone and quasidense, $\VT$ be an open convex subset of $E^*$, $\VT \ni 0$, $\VT \cap R(T) \neq \emptyset$ and
\begin{equation}\label{FPALT1}
(t,t^*) \in G(T)\ \hbox{and}\ t^* \in \VT \qlr \bra{t}{t^*} \ge 0.
\end{equation}
Then $(0,0) \in G(T)$.  
\end{lemma}
\begin{proof}
Let $y^* \in \VT \cap R(T)$.   Since the segment $[0,y^*]$ is a compact subset of the open set $\VT$, we can choose $\eps > 0$ so that\quad$\KT := [0,y^*] + \eps E^*_1 \subset \VT$.\quad
From \eqref{DNORMAL3},  $R(T) \cap \intr\,R(\partial\tauKT) = R(T) \cap \intr\,\KT \ni y^*$.   Using \Lem{Tlem}, there exist  $z^* \in \KT$ and $z\dbs,x\dbs \in E\dbs$ satisfying \eqref{T2}.  
\par
From \Lem{EXSEQlem}, there exists a sequence $\{(t_n,t_n^*)\}_{n \ge 1}$ of elements of $G(T)$ such that\quad$\bra{t_n}{t_n^*} \to \bra{z^*}{z\dbs}$\quad and\quad $\|t_n^* - z^*\| \to 0$ as $n \to \infty$.\quad If $n$ is\break sufficiently large, $t_n^* \in \VT$, and so, from \eqref{FPALT1}, $\bra{t_n}{t_n^*} \ge 0$.   Passing to the limit, $\bra{z^*}{z\dbs} \ge 0$, and so \eqref{T2} now implies that
\begin{equation*}
\sup\bra{\KT}{x\dbs - z\dbs} + \half\|(z^*,x\dbs)\|^2 \le 0.
\end{equation*}  
Since\quad $\sup\bra{\KT}{x\dbs - z\dbs} \ge \eps\|x\dbs - z\dbs\|$,\quad it follows that $z\dbs = x\dbs$, $z^* = 0$ and $x\dbs = 0$.   Consequently, $0 = z\dbs \in T^\F(z^*) = T^\F(0)$, and \eqref{FITZ2} implies that $(0,0) \in G(T)$. 
\end{proof}
\begin{remark}
In the context of \Lem{Tlem}, one can in fact prove that\quad $\bra{z^*}{x\dbs} + \half\|(z^*,x\dbs)\|^2 = 0$\quand
$\bra{z^*}{x\dbs - z\dbs} = \sup\bra{\KT}{x\dbs - z\dbs}$.
\end{remark}
\par
We now sketch the ``dual'' version of \Lem{Tlem}, \Lem{Ulem}, which can be used to prove \Thm{FPVthm}: however this proof of \Thm{FPVthm} does\break require \Lem{EXSEQlem}, and the appearance of $\ddot K$ makes \Lem{Ulem} innately more ``technical'' than \Lem{Tlem}.
\begin{lemma}\label{Ulem}
Let $T\colon\ E \toto E^*$ be closed, monotone and quasidense and $K$ be a bounded closed convex subset of $E$ such that $D(T) \cap \intr K \ne \emptyset$.   Let $\ddot K$ be the $w(E\dbs,E^*)$--closure of $\wh K$ in $E\dbs$.   Then there exist $z\dbs \in \ddot K$ and $z^*,v^* \in E^*$ such that
\begin{equation*}
z\dbs \in T^\F(z^*)\hbox{ and }\bra{z^*}{z\dbs} + \sup\bra{K}{v^* - z^*} + \half\|(v^*,z\dbs)\|^2 \le 0.\label{U2}
\end{equation*}  
\end{lemma}
\begin{proof}
For all $(x,x^*) \in E \times E^*$, let
\begin{equation*}
f(x,x^*) :=
\begin{cases}
\inf_{y^* \in E^*}\big[\varphi_T(x,y^*) + \sup\bra{K}{x^* - y^*}\big]
&(x \in K);\\
\infty
&(x \not\in K).
\end{cases}
\end{equation*}
Then $f$ is proper and convex, for all $(v^*,z\dbs) \in E^* \times E\dbs$,
\begin{equation*}
f^*(v^*,z\dbs) = \minn_{z^* \in E^*}\big[{\varphi_T}^*(z^*,z\dbs) + \sup\bra{K}{v^* - z^*}\big] + I_{\ddot K}(z\dbs)\end{equation*}
and, for all $(x,x^*) \in E \times E^*$,\quad 
$f(x,x^*) + \half\|(x,x^*)\|^2 \ge 0$.\quad The result now follows from Rockafellar's version of the Fenchel duality theorem, just as in \Lem{Tlem}.
\end{proof}

\end{document}